%% file: NLHPML.tex
\documentclass[leqno,final]{article}

\usepackage{fancyhdr} 
\usepackage{float} 
\usepackage{tikz}  
\usepackage{graphicx}
\usepackage{epstopdf}
\usepackage{caption}
\usepackage{subfig}
\usepackage{algorithmic}
\ifpdf
  \DeclareGraphicsExtensions{.eps,.pdf,.png,.jpg}
\else
  \DeclareGraphicsExtensions{.eps}
\fi

\usepackage{cancel} 
\usepackage{amsmath,amssymb,amsthm,amsfonts}
\newtheorem{theorem}{Theorem}[section]
\newtheorem{lemma}[theorem]{Lemma}
\newtheorem{remark}[theorem]{Remark}
\newtheorem{corollary}[theorem]{Corollary}

\numberwithin{theorem}{section}
\numberwithin{equation}{section}
\numberwithin{figure}{section}
\numberwithin{table}{section}

\usepackage{mathrsfs}
\usepackage{bm}
\usepackage{enumerate}
\usepackage[normalem]{ulem}
\let\isout\sout \renewcommand{\sout}[1]{\ifmmode\text{\isout{\ensuremath{#1}}}\else\isout{#1}\fi}
\usepackage{diagbox}

\usepackage{color}
\usepackage{xcolor}
\colorlet{inlinkcolor}{green!50!black}
\colorlet{exlinkcolor}{red!50!black}
\usepackage[colorlinks = true,
  allcolors = inlinkcolor,
  urlcolor = exlinkcolor,
            ]{hyperref}
            
\usepackage{indentfirst}
\usepackage{booktabs} % for tabular
\usepackage{multirow}  % for tabular
\usepackage{diagbox}
\newcommand\myslbox{\diagbox[width=4.7em,height=\line]{\phantom{x}}{\phantom{x}}}
\input{NLHPML_def}  %User definitions

\title{Finite Element Method for a Nonlinear PML Helmholtz Equation \\ 
with High Wave Number}
\newcommand{\email}[1]{\protect\href{mailto:#1}{#1}}
\author{
  Run Jiang\thanks{Department of Mathematics, Nanjing University, Jiangsu, 210093, People's Republic of China (\email{dz1821002@smail.nju.edu.cn}, \email{liyonglin@smail.nju.edu.cn}, \email{hjw@nju.edu.cn}). This work 
  of these two authors was partially supported by the NSF of China under grants 12171238 and 11525103.}
  ,~~Yonglin Li\footnotemark[1] \thanks{Institute of Computational Mathematics and Scientific/Engineering Computing, Academy of Mathematics and Systems Science, Chinese Academy of Sciences, Beijing, 100190. 
Department of Applied Mathematics, The Hong Kong Polytechnic University, Hung Hom, Hong Kong. People's Republic of China. The work of YL was partially supported by CAS AMSS-PolyU Joint Laboratory of Applied Mathematics.
}
  ,~~Haijun Wu\footnotemark[1]
  ~~and~~Jun Zou\thanks{Department of Mathematics, The Chinese University of Hong Kong, Shatin, N.T., Hong Kong, P. R. China (\email{zou@math.cuhk.edu.hk}). 
The work of JZ was substantially supported by Hong Kong RGC General Research Fund (projects 
14306719 and 14306718).}
}

\date{}  %%date
\begin{document}
\maketitle

%% abstract
\begin{abstract}
  A nonlinear Helmholtz equation (NLH) with high wave number and Sommerfeld radiation condition is 
 approximated by the perfectly matched layer (PML) technique and 
 then discretized by the linear finite element method (FEM).
 Wave-number-explicit stability and regularity estimates and the exponential convergence  are proved for the nonlinear truncated PML problem.
Preasymptotic error estimates are obtained for the FEM, where  the logarithmic factors in $h$ required by the previous results for the NLH with impedance boundary condition are removed in the case of two dimensions. Moreover, local quadratic convergences of the Newton's methods are derived for both the NLH with PML and its FEM. 
Numerical examples are presented to verify the accuracy of the FEM, which demonstrate 
that the pollution errors may be greatly reduced by applying the interior penalty technique with proper penalty parameters to the FEM. The nonlinear phenomenon of optical bistability can be successfully simulated.
\end{abstract}

%% Key words
{\bf Key words.} 
  Nonlinear Helmholtz equation, high wave number, perfectly matched layer, Newton's method,  finite element method, preasymptotic error estimates.

{\bf AMS subject classifications. }
65N12, %Stability and convergence of numerical methods
65N15, %Error bounds
65N30, %finite elements, Rayleigh-Ritz and Galerkin methods, finite methods
78A40  %Wave and radiation
 
%%%%%%%%%%%%%% Documents %%%%%%%%%%%%
\section{Introduction}\label{sec:Introduction}
We are mainly concerned in this work with the following nonlinear Helmholtz equation (NLH) which may model 
some optical wave scattering 
by a nonlinear medium with a Kerr-type nonlinearity \cite{boyd2008,wuzou2018,fi2001}: 
\begin{alignat}{2}
-\Delta u - k^2 u - k^2\vep\oneo\abs{u+\uinc}^2(u+\uinc) &= f &\quad &\text{in } \R^{d}\;(d=2,3), \label{eq:Helm}\\
\abs{\frac{\partial u}{\partial r} - \i k u} &= o(r^{\frac{1-d}2}) &\quad &\text{as } r=\abs{x} \to \infty, \label{eq:Somm}
\end{alignat}
%\begin{alignat}{2}
%-\Delta u - k^2(1+\vep\oneo\abs{u}^2)u &= f &\quad &\text{in } \R^{d}, \label{eq:Helm1}\\
%\abs{\frac{\partial (u-\uinc)}{\partial r} - \i k(u-\uinc)} &= o(r^{\frac{1-d}2}) & &\text{for } r = \abs{x}\to \infty, \label{eq:Somm1}
%\end{alignat}
where the scattered wave $u$ is a component of the electric field,  $\uinc$ denotes the incident wave, $k\gg1$ is the wave number, $\Omz\subset\Om$ is the region occupied by the Kerr medium, $\oneo$ is the characteristic function of $\Omz$, and
$\vep$ is called the Kerr constant satisfying $0<\vep\ll 1$, 
defined by $\vep=4n_2/n_0$ with $n_0$ and $n_2$ to be the linear and the second-order indices of refraction, respectively. Both $n_0$ and $n_2$ are assumed real so that the medium is transparent or lossless. \eqref{eq:Somm} is the Sommerfeld radiation condition, which ensures that the scattered wave 
is only outgoing. 
 In general, $f: \R^d\to \mathbb{C}$ is an $L^2$ function depending on the source and the incident wave, namely, $f=f_0 + \Delta \uinc+k^2\uinc$ with some source term $f_0$. Obviously, the total field $U:=u+\uinc$ satisfies the equation 
\begin{equation}\label{eq:Helm1}
-\Delta U - k^2 U - k^2\vep\oneo\abs{U}^2 U = f_0 \quad \mbox{in }\R^d. 
\end{equation}
We suppose that $f$ is compactly supported, that is, $\supp f\subset\Om$,  where $\Om = \B_R$ is a ball centered at the origin with radius $R$.  For simplicity, we assume that the wave number $k$ is constant in the whole space $\R^d$.
We write $\Gamma:=\partial\Om$, and often have 
$\dist(\Ga,\Omz)\geq C\diam(\Omz)$ for some constant $C$.

For numerical solutions, we should approximate the system \eqref{eq:Helm}--\eqref{eq:Somm} on a bounded domain.
The PML technique is an efficient and very popular mesh termination technique in computational wave propagation, which was originally proposed by B{\'e}renger \cite{berenger1994}. The key idea is to surround $\Om$ by a specially designed layer which can strongly absorb the outgoing waves entering the layer. Since the outgoing waves are strongly absorbed by PML,  
it is natural to truncate the scattered field by the simplest homogeneous Dirichlet boundary condition after an appropriate distance from the region $\Omega$,
say, at $r=\hR$ for some $\hR>R$, as the outgoing waves would be sufficiently small there. 
For the linear Helmholtz equation, existing studies (see, e.g., 
\cite{chen2005,chenwu2003,bramble2007,liwu2019,baowu2005,lassas1998}) indicate that the truncated PML solution converges exponentially when the width of the layer or the PML parameter tends to infinity.  
In particular, Li and Wu \cite{liwu2019} proved some wave-number-explicit stability and convergence estimates for the linear Helmholtz equation with truncated PML on the whole computational domain  $\D:=\B_{\hR}$ including the PML region between $\Gamma$ and $\hGamma:=\partial \D$  (see Figure~\ref{fig:region}).
\begin{figure}[tbp]
\centering
\begin{tikzpicture}
  \draw [red, thin] (0,2) circle [radius=1.6];
  \draw [blue, thin] (0,2) circle [radius=2.2];
  \draw [fill=cyan!20] (-0.75,2) to [out=90,in=180] (-0.25,2.5)
  to [out=0,in=180] (1,2.2) to [out=0,in=10] (0.5,1.5) 
  to [out=190 ,in=270] (-0.75,2);
  \node [scale=0.8]at (0,2) {$\Omega_0$};
  \node [scale=0.8]at (0,1) {$\Omega$};
  \node [scale=0.8]at (0,0.15) {PML};
  \node [scale=0.8]at (2.2,0.4) {$\D = \B_{\hR}$};
  \node [scale=0.8]at (1.75,2) {$\Gamma$};
  \node [scale=0.8]at (2.35,2.05) {$\hat{\Gamma}$};
\end{tikzpicture}
\caption{\itshape{Setting of the nonlinear PML problem.}}
\label{fig:region}
\end{figure}
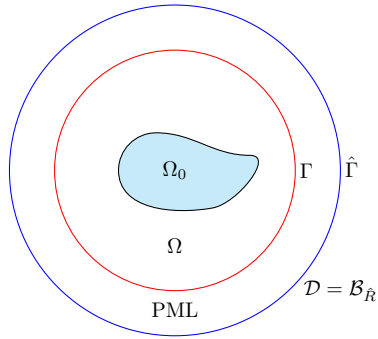

It is well-known that when solving the wave scattering problems in high frequency, the FEMs of fixed order may  suffer the so-called pollution effect, that is, its performance decreases as the wave number increases \cite{babuska2000}. It is of significance in the theory and practical applications of FEMs to derive an error estimate containing the pollution error, namely, the preasymptotic error estimate. For preasymptotic error estimates of FEMs for the linear Helmholtz equations we refer to \cite{ib95a,ib97,zhuwu2013,wu2013,duwu2015} for the impedance boundary condition and \cite{liwu2019,cgnt22} for the PML boundary condition. Melenk and Sauter \cite{melenk2010,ms11} showed that the $hp$-FEM is pollution free if its order is allowed to vary with the wave number $k$ (i.e. proportional to $\ln k$). 

Contrary to the aforementioned rich references for linear Helmholtz equations with high wave number, 
we are not aware of rigorous mathematical and finite element studies of the NLH system \eqref{eq:Helm}
in the literature. We were the first time to carry out in \cite{wuzou2018} 
a systematical mathematical and numerical study of the NLH system with impedance boundary condition.
The well-posedness of both the NLH system and its linear finite element approximation was established. Particularly, 
the stability estimates of the continuous NLH solutions and their finite element solutions were 
achieved with explicit dependence on the wave number, and the preasymptotic optimal error estimates of the finite element solutions were also derived. 

The purpose of this paper is to extend the results in our early work in \cite{wuzou2018} to 
the practically more important case, i.e., the NLH \eqref{eq:Helm} with PML boundary condition. 
Note that PML is a  much more accurate approximation to the radiation condition \eqref{eq:Somm} than the impedance boundary condition so that we can use smaller computational domain to truncate the unbounded domain and hence significantly save the computational cost. 
%Similarly to \cite{wuzou2018}, 
Our key idea is to introduce the Newton's sequences of approximate linearized problems to the continuous NLH problem with PML and its FEM, respectively, and then establish the convergence of the two sequences 
and the preasymptotic error estimates between them. It is noted that those estimates in \cite{wuzou2018} are based on the simplest iteration, that is, the frozen-nonlinearity method, while in this paper, we consider the Newton's method \cite{yuan2017} and give its corresponding estimates, in particular, its quadratic convergence. Specifically,
for the NLH with PML, we shall derive the wave-number-explicit stability and regularity estimates as well as the exponential convergence of its solution, under the condition that $\max\big\{k^{d-2}\vep\Mf^2,~k\vep\Linf{\uinc}^2 \big\}$ is sufficiently small and some other mild conditions on the PML parameters (see \eqref{eq:assumption}), where $\Mf= \norm{f}_{0,\Omega}+k^2\vep\norm{\uinc}_{L^6(\Om_0)}^3$. Furthermore, we establish the stability and preasymptotic error estimates 
when the linear FEM is used to approximate the NLH with PML, 
under the conditions that $k^3h^2$ and $\max\big\{ k^{d-2}\vep\abs{\ln h}^{2\bar{d}}\Mf^2,~k\vep\Linf{\uinc}^2\big\}$ are sufficiently small and the same conditions \eqref{eq:assumption} on the PML parameters, 
where $\bar d=0$ for $d=2$ and $\bar d=1$ for $d=3$.  The fact that $\bar d=0$ for $d=2$ indicates 
the condition on $\vep$ for the FEM do not contain a logarithmic factor in $h$ in two dimensions, 
which is the same as that for the original NLH and improves the condition 
in our previous work \cite{wuzou2018} 
with impedance boundary condition. Moreover, we present numerical examples to verify the accuracy of the FEM, most importantly, to demonstrate that the pollution error may be greatly reduced by applying the continuous interior penalty finite element method (CIP-FEM) %interior penalty technique 
\cite{douglas1976,zhuwu2013,wu2013,duwu2015,liwu2019} %to the FEM (i.e. CIP-FEM)  
and selecting proper penalty parameters, as well as to successfully simulate the nonlinear optical phenomenon of optical bistability (see \cite{boyd2008}) by using the CIP-FEM solved by the Newton's method.

The rest of this paper is organized as follows. In Section~\ref{s:2}, we introduce the nonlinear truncated PML problem for the NLH and three iterative methods for solving the PML system.  %\sout{recall wave-number-explicit stability and regularity estimates from \cite{liwu2019} for the linear truncated PML problem.}  
Section~\ref{s:3} is devoted to %wave-number-explicit 
the stability estimates and the exponential convergence of the approximate solution to the nonlinear truncated PML problem. The quadratic convergence of the Newton's iteration for the nonlinear truncated PML problem is also achieved. 
In Section~\ref{s:4}, we establish the preasymptotic error estimates of the FEM for the nonlinear truncated PML problem and the quadratic convergence of the Newton's iteration for the nonlinear FEM, and further introduce the CIP-FEM %continuous interior penalty finite element method (CIP-FEM) 
to reduce the pollution error.
In Section~\ref{s:5}, some numerical examples are provided to verify the accuracies of the FEM and CIP-FEM, especially to recover the phenomenon of optical bistability. 
 
Throughout the paper, $C$ is used to denote a generic positive constant that is independent of $h,\,k,\,f$, and the penalty parameters, but may depend on the PML absorbing parameter $\sigma_0$ and thickness $L$ at most polynomially. We also use the shorthand notations $A\ls B$ and $B\gtrsim A$ for the inequality $A\leq CB$. $A\eqsim B$ is a notation for the statement that $A\ls B$ and $A\gtrsim B$. In addition, some standard Sobolev spaces, norms and inner products associated with Helmholtz equations 
are adopted, as in \cite{brenner2008,ciarlet1978}. In particular, $(\cdot,\cdot)_Q$ and $\ine{\cdot}{\cdot}$ denote the $L^2$-inner product on complex-valued $L^2(Q)$ and $L^2(e)$ spaces, respectively. 
For simplicity, we will write by $\norm{\cdot}_{s,G}$ and $\abs{\cdot}_{s,G}$ the norm and semi-norm of the Sobolev space $H^s(G)$ for any domain $G\subset \R^d$, and 
write by $\chi_G$ the characteristic function of $G$.

%%%%%%%%%%%%%%%%%%%%%%%%%%%%%%%%%%%%%%%%%%%%%%%%%%%%%%
\section{The approximate PML problem}\label{s:2}
In this section we approximate the NLH \eqref{eq:Helm}--\eqref{eq:Somm} by the PML technique %, then recall some stability estimates for the linear Helmholtz problem with PML, and lastly,
and state three iterative methods for the derived nonlinear PML problem.

\subsection{The nonlinear approximate PML problem}
It is well known that the PML system can be viewed as a consequence of the original scattering problem 
by a complex coordinate stretching (see e.g. \cite{chew1997,collino1998}). 
For simplicity, we consider the circular/spherical PML with constant absorbing coefficient. Let
\begin{align}\label{tr}
\tr:=\intr\al(s) ds=r\be(r),  \quad \text{with } \al(r)=1+\i\si (r), ~\be(r)=1+\i \de(r),
\end{align}
where $\si(r)$ and $\de(r)$ are given by 
\begin{equation} \label{def:medium}
\si (r)=\left\{
\begin{aligned}
& 0, & 0\leq r \leq R, \\
& \siz,  & r>R,
\end{aligned}
\right. \qquad \de(r)=\left\{
\begin{aligned}
& 0, & 0\leq r \leq R, \\
& \frac{\siz(r-R)}{r},  & r>R, 
\end{aligned}
\right.
\end{equation}
with $\siz>0$ being a constant. We assume that the PML medium property $\si$ is constant here to  
simplify the theoretical analysis, even though it is possible to employ the variable PML medium properties 
in practice, e.g., the PML parameter can be chosen as $\si (r)=\sigma_0 (\hR-R)^{-m}(r-R)^m$ with $m\geq 1$ when $R < r \leq \hR$. However, the theoretical analysis of variable PML medium properties 
will be much more technical and not be considered in this work. 
The PML equation is obtained from the Helmholtz equation \eqref{eq:Helm} by replacing the radial coordinate $r$ by $\tr$. For example, in the case of two dimensions $(d=2)$, the Helmholtz equation \eqref{eq:Helm} can be written in polar coordinates  as follows: 
\begin{equation}\label{eq:Helm:polar}
-\frac{1}{r}\frac{\partial}{\partial r} \left(r\frac{\partial u}{\partial r}\right) - \frac{1}{r^2}\frac{\partial^2 u}{\partial \theta^2}-k^2u - k^2\vep\oneo\abs{u+\uinc}^2 (u+\uinc) = f.		
\end{equation}
Then the PML equation is given by 
\[ -\frac{1}{\tr}\frac{\partial}{\partial \tr} \left(\tr\frac{\partial \tu}{\partial \tr}\right) - \frac{1}{\tr^2}\frac{\partial^2 \tu}{\partial \theta^2}-k^2\tu - k^2\vep\oneo\abs{\tu+\uinc}^2 (\tu+\uinc) = f,
\]
where $\tu(r,\theta):=u(\tr,\theta)$. Noting that ${\partial}/{\partial\tr}=\al^{-1} {\partial}/{\partial r}$ and $\tr=\be r$, we get
\begin{equation*}
-\frac{1}{r}\frac{\partial}{\partial r} \left(\frac{\be r}{\al}\frac{\partial \tu}{\partial r}\right) - \frac{\al}{\be r^2}\frac{\partial^2 \tu}{\partial \theta^2}-\al\be k^2\tu - k^2\vep \oneo\abs{\tu+\uinc}^2 (\tu+\uinc) = f.
\end{equation*}
We note that $\tu = u$ in $\Om$ and is expected to decay exponentially away from $\Gamma$. Therefore the PML problem is truncated at $r=\hR$, 
where $\tu$ is sufficiently small. Let $\hOm = \{ x\in \R^d:\abs{x}\in (R,\hR) \}$ and $L:=\hR-R$ denote the PML domain and its thickness, respectively. Recalling the notation $\D = \B_{\hR}$ and $\hat\Gamma = \partial \D$, we arrive at the following nonlinear truncated PML problem:
\begin{equation} \label{PML:2d}
-\frac{1}{r}\frac{\partial}{\partial r} \left(\frac{\be r}{\al}\frac{\partial \hu}{\partial r}\right) - \frac{\al}{\be r^2}\frac{\partial^2 \hu}{\partial \theta^2}-\al\be k^2\hu - k^2\vep\oneo\abs{\hu+\uinc}^2 (\hu+\uinc) = f \quad  \mbox{in } \D;\quad\, \hu =0 \quad \mbox{on } \hat{\Gamma}. 
\end{equation}
The nonlinear PML problem for three dimensional case in spherical coordinates 
can be derived in a similar way and written as (see \cite{liwu2019} for details):
\begin{equation} \label{PML:3d}
-\frac{1}{r^2} \frac{\partial}{\partial r} \left(\frac{\be^2 r^2}{\al}\frac{\partial \hu}{\partial r}\right) - \frac{\al}{r^2}\Delta_S \hu  - \al\be^2 k^2\hu - k^2\vep\oneo\abs{\hu+\uinc}^2 (\hu+\uinc) = f \quad \mbox{in  } \D; \quad\, \hu = 0 \quad  \mbox{on } \hat{\Gamma},
\end{equation}
where $\Delta_S = \frac{1}{\sin\theta}\frac{\partial}{\partial\theta} \left(\sin\theta\frac{\partial}{\partial\theta}\right) + \frac{1}{\sin^2\theta}\frac{\partial^2}{\partial\vp ^2}$ is the Laplace-Beltrami operator on the unit sphere.

In Cartesian coordinates, we denote by $\hat L$ the linear differential operator:
\begin{equation*}%\label{eq:hL}
\hat L w:=-\nabla\cdot (A\nabla w) - Bk^2 w,
\end{equation*}
where $A$ and $B$ are defined as 
$ %\begin{align*}
A=HDH^T$,  $B=\al(r)\be^{d-1}(r),
$ %\end{align*}
with the matrices $D$ and $H$ given by 
\[
\begin{array}{cc}
D=\begin{pmatrix}
	\frac{\be(r)}{\al(r)} & 0 \\
	0 & \frac{\al(r)}{\be(r)}	
  \end{pmatrix},~
H=\begin{pmatrix}
	\cos\theta & -\sin\theta \\
	\sin\theta & \cos\theta
  \end{pmatrix} & \text{for } d=2,\\
D=\begin{pmatrix}
	\frac{\be^2(r)}{\al(r)} & 0      & 0    \\
	0                       & \al(r) & 0    \\
	0                       & 0      & \al(r)
  \end{pmatrix},~
H=\begin{pmatrix}
	\sin\theta\cos\vp & \cos\theta\cos\vp & -\sin\vp \\
	\sin\theta\sin\vp & \cos\theta\sin\vp & \cos\vp \\
	\cos\theta            & -\sin\theta           & 0 
  \end{pmatrix} & \text{for } d=3\,.
\end{array}
\]
Then the nonlinear PML problems \eqref{PML:2d} and \eqref{PML:3d} can be rewritten in the unified form in $\R^d$:  %the nonlinear PML problem \eqref{eq:PML:all} can be simply rewritten as 
\begin{equation}\label{eq:PML}
\hat L \hu - k^2 \vep \oneo \abs{\hu+\uinc}^2(\hu+\uinc) = f\quad\mbox{in } \D;\quad\, \hu=0\quad\mbox{on } \hat{\Gamma}. 
\end{equation}
For simplicity, throughout the rest of the paper, we shall use the notations:
\[\norm{\cdot}_s = \norm{\cdot}_{s,\D} = \norm{\cdot}_{H^s(\D)}, \quad \abs{\cdot}_s = \abs{\cdot}_{s,\D} = \abs{\cdot}_{H^s(\D)} \qaq (\cdot,\cdot)=(\cdot,\cdot)_{\D}.
\] 
Since $A(x)$ is discontinuous across $\Gamma = \partial\Omega$, $\hu$ may be not in the space $H^2(D)$. Note that 
\[\hu\in H^2(\OchO):=\{v\in L^2(\D) : v|_\Om\in H^2(\Om),\; v|_{\hat\Om}\in H^2(\hat\Om)\},\]
we define the corresponding norm and semi-norm by
\[\Ht{\cdot}{\OchO}=\big(\Ht{\cdot}{\Om}^2+\Ht{\cdot}{\hOm}^2\big)^{1/2},\quad \sHt{\cdot}{\OchO}=\big(\sHt{\cdot}{\Om}^2+\sHt{\cdot}{\hOm}^2\big)^{1/2}.\]

The variational formulation of the nonlinear PML problem \eqref{eq:PML} %--\eqref{eq:PML:bound:all} 
reads as: find $\hu\in H_0^1(\D)$ such that
\begin{equation}\label{varnltpml}
 \anl(\hu,v) = (f,v) \quad\forall v\in H_0^1(\D),
\end{equation}
where $\anl(u,v)$ is defined by 
\begin{align}
\anl(u,v) &:= a(u,v)-k^2\vep\big(\abs{u+\uinc}^2(u+\uinc),v\big)_{\Omz},\label{a1} \\
a(u,v) &:= (A\na u,\na v)-k^2(Bu,v). \label{a2}
\end{align}

We shall often use the following energy norm in the subsequent analysis:
\begin{equation} \label{He}
\He{v}=\Big(\Re \big(a(v,v)\big)+2k^2\LtD{v}^2\Big)^{1/2} \quad \forall\,v\in H^1(\D)\,.
\end{equation}
It can be shown that $\He{v}\eqsim k\LtD{v}+\LtD{\na v}$. In fact, noting that $0\leq \de \leq \si \leq \siz$, 
we have for the case of $d=2$ that 
\begin{align*}
\He{v}^2 &= \Re \big(a(v,v)\big)+2k^2\LtD{v}^2 \\
 &= \int_0^{2\pi} {\hskip -7pt} \inthR \bigg( \frac{1+\si\de}{1+\si^2} r\abs{v_r}^2 +\frac{1+\si\de}{1+\de^2} \frac 1r \abs{v_\theta}^2 +(1+\si\de)k^2 r\abs{v}^2 \bigg)dr d\theta, 
\end{align*}
which leads to  
\begin{equation} \label{h12d} 
(1+\siz^2)^{-1} \LtD{\na v}^2 +k^2\LtD{v}^2 \leq \He{v}^2 \leq (1+\siz^2) (\LtD{\na v}^2 + k^2\LtD{v}^2).
\end{equation}
Similarly, we obtain for the case of $d=3$ that 
\begin{align*}
&\He{v}^2 = \Re \big(a(v,v)\big)+2k^2\LtD{v}^2 \\
& = \int_0^{2\pi}{\hskip -7pt} \int_0^\pi {\hskip -5pt} \inthR \sin\theta \bigg( \frac{1-\de^2+2\si\de}{1+\si^2} r^2\abs{v_r}^2 + \abs{v_\theta}^2 + \frac{1}{\sin^2\theta}\abs{v_\varphi}^2 +(1+\de^2+2\si\de)k^2 r^2\abs{v}^2 \bigg) dr d\vp d\theta ,
\end{align*}
which leads to
\begin{equation} \label{h13d} 
(1+\siz^2)^{-1} \LtD{\na v}^2 +k^2\LtD{v}^2 \leq \He{v}^2 \leq (1+3\siz^2) (\LtD{\na v}^2 + k^2\LtD{v}^2).
\end{equation}

\subsection{Iterative methods for the nonlinear PML problem}
To solve the nonlinear PML problem \eqref{eq:PML}, we introduce three iterative methods. 
The first and simplest one is the frozen-nonlinearity iteration: 

Given initial function $\hat u_0\in H_0^1(\D)$, find $\hu^{l+1}\in H_0^1(\D)$ for $l=0,1,2,\cdots$, such that
\begin{equation}\label{eq:fp}
\hat L \hu^{l+1} - k^2 \vep \oneo \abs{\hu^l + \uinc}^2 (\hu^{l+1}+\uinc) = f.
\end{equation}
It is known that the iteration \eqref{eq:fp} only has the linear convergence rate and converges for problems with weak nonlinearity (see \cite{yuan2017,wuzou2018}). The analysis of the nonlinear PML problem \eqref{eq:PML} based on the iteration \eqref{eq:fp} is similar to that of the NLH \eqref{eq:Helm1} with impedance boundary condition in \cite{wuzou2018} and is omitted here.

The second one is the Newton's method: 

Given $\hu^0\in H_0^1(\D)$, find $\hu^{l+1}\in H_0^1(\D)$ for $l=0,1,2,\cdots$, such that 
\begin{equation}\label{ip1}
\begin{aligned}
\hat L \hu^{l+1} &- k^2 \vep \oneo \BgS{2\abs{\hu^{l}+\uinc}^2 \hu^{l+1} + \big(\hu^{l}+\uinc\big)^2 \overline{\hu^{l+1}}} \\
= f &- k^2 \vep \oneo \BgS{2\abs{\hu^{l} + \uinc}^2 \hu^{l}-\big(\hu^{l}+\uinc\big)^2\overline{\uinc}}.
\end{aligned}
\end{equation}
The Newton's method  converges not only for problems with weak nonlinearity but also for problems with strong nonlinearity and converges at a quadratic rate once the initial function $\hat u^0$ is sufficiently close to the exact solution. The Newton's method will be analyzed in the next section.
 
The third one is a modified Newton’s method proposed by \cite{yuan2017}. It is obtained by replacing $\overline{\hu^{l+1}}$ in \eqref{ip1} by $\overline{\hu^{l}}$ and then given by 
\begin{equation}\label{eq:mN}
\begin{aligned}
\hat L \hu^{l+1} - 2k^2 \vep \oneo \abs{\hu^{l}+\uinc}^2 \hu^{l+1} = f - k^2 \vep \oneo \abs{\hu^{l}+\uinc}^2 \Sp{\hu^{l} - \uinc}, \quad l\geq 0.
\end{aligned}
\end{equation}
Compared with the Newton's method, the modified Newton's method has only linear convergence rate but numerical evidences indicate that it is robust with respect to the initial guess. The analysis of this method is left to a future work.   

\section{Analyses of the nonlinear PML problem}\label{s:3}
In this section, we shall present the well-posedness of the nonlinear PML problem \eqref{eq:PML} and 
prove the exponential convergence of the nonlinear PML solution to the original NLH solution. To do so, we regard the nonlinear PML solution as the limits of the sequence constructed by Newton's iteration \eqref{ip1}. %a special linearization procedure (cf. \cite{wuzou2018}). 
We shall first derive some uniform bounds for the linearized problems and then prove the quadratic convergence of the iteration sequence.

\subsection{An auxiliary linearized problem}\label{au-problem}
Before analyzing the nonlinear PML problem \eqref{eq:PML}, we study a 
linearized problem associated with the Newton's iteration \eqref{ip1}: 
for given function $\phi\in L^{\infty}(\D)$ and $g\in L^2(\D)$, $\hat w^\phi \in H^1_0(\D)$ solves
\begin{equation}\label{ap1}
\hat L \hat w^\phi-k^2\vep \oneo \bgS{2\abs{\phi+\uinc}^2\hat w^\phi + (\phi + \uinc)^2 \overline{\hat w^\phi}} = g. 
\end{equation}

For the stability of the solution to this auxiliary linear system, we first recall some estimates 
of the solution to a linear Helmholtz problem and the Hankel functions of the first kind.
%\subsection{The linear PML problem}
%We shall study the nonlinear PML system \cb{\eqref{eq:PML}} through its linear form 
%(with a given source $g\in L^2(\D)$): 
%%consider the linear PML problem associated 
%%with the nonlinear PML system \cb{\eqref{eq:PML}}: %\eqref{eq:PML:bound:all}
%\begin{alignat}{2}
%-\na\cdot(A\na\hat w)-Bk^2\hat w&=g &\quad &\text{in }\D; 
%\quad ~\hat w=0 \quad \text{on } ~\hat\Gamma, \label{LTPML1} 
%%\hat w&=0 &\quad &\text{on }\hat\Gamma.\label{LTPML2}
%\end{alignat}
%for which we know the following stability and regularity estimates. 
%
\begin{lemma}[{\cite[Theorem 3.1 and Corollaries 3.4 and 3.9]{liwu2019}}]\label{lem:stabl}
For a given source $\hat g\in L^2(\D)$, let $\hat w\in H_0^1(\D)$ solve $\hat L \hat w = \hat g$, %be the solution to the linear system: 
%\begin{alignat}{2}
%-\na\cdot(A\na\hat w)-Bk^2\hat w&=\hat g &\quad &\text{in }\D; 
%\quad\, \hat w=0 \quad \text{on } ~\hat\Gamma. \label{LTPML1} 
%%\hat w&=0 &\quad &\text{on }\hat\Gamma.\label{LTPML2}
%\end{alignat}
then under the conditions that $R\eqsim\hR\eqsim 1$ and  
\begin{equation}\label{eq:assumption}
kR\geq 1 \qaq k\siz L\geq\max\big\{ 2kR+\sqrt{3}kL,10 \big\}, 
\end{equation}
there exists a positive constant $C_{\mL}$ independent of $k$ and $\hat g$ such that 
\begin{equation}\label{stab:LTPML} 
k\LtD{\hat w}+\He{\hat w}+k^{-1}\sHt{\hat w}{\OchO} \leq C_{\mL} \LtD{\hat g}.
\end{equation}
\end{lemma}

%For further analysis, we shall need the following property of the Hankel functions of the first kind.
\begin{lemma}[{\cite[Lemma 2.2]{chen2005}}]\label{lem:Bessel}
For any $\nu\in\R$, $z\in\Cm_{++}=\left\{z\in \Cm:\Im(z)\geq 0, \Re(z)\geq 0\right\}$ and $0<x\leq\abs{z}$, 
the following estimate holds for the Hankel function $H_{\nu}^{(1)}(z)$ of the first kind:
\begin{equation}\label{eq:Hankel:bound}
\abs{H_{\nu}^{(1)}(z)} \leq e^{-\Im(z)\left(1-\frac{x^2}{\abs{z}^2}\right)^{1/2}}\abs{H_{\nu}^{(1)}(x)}.
\end{equation}
%In addition, for any $n\in\Z,~\l\in\N,~z\in\C$, there holds 
%\begin{equation}\label{eq:Bessel:bound}
%\abs{\J(z)}\leq e^{\abs{\Im(z)}},\quad \abs{J_{\l+\frac12}(z)}\leq \abs{\frac{2z}{\pi}}^{\frac12}e^{\abs{\Im(z)}}.
%\end{equation}
\end{lemma}

Using Lemma \ref{lem:stabl}, we can readily get the stability estimate of the solutions to the auxiliary linear system \eqref{ap1}. To do so, we rewrite it as 
\begin{equation}\label{eL1}
\hat L \hat w^\phi = g + k^2\vep \oneo \bgS{2\abs{\phi+\uinc}^2\hat w^\phi + (\phi + \uinc)^2 \overline{\hat w^\phi}}.
\end{equation}
%\begin{alignat}{2}
%-\na\cdot(A\na \hat w^\phi)-Bk^2 \hat w^\phi&=g+k^2\vep \oneo\abs{\phi+\uinc}^2\hat w^\phi \quad & \mbox{in }\D;
%\quad \hat w^\phi = 0 \quad  \mbox{on } ~\hat{\Gamma}. \label{eL1}
%%\hat w^\phi&=0\quad & \mbox{on }\hat\Gamma. \label{eL2}
%\end{alignat}
%Now applying Lemma \ref{lem:stabl} to the system above 
%with $\hat g=g+k^2\vep \oneo\abs{\phi+\uinc}^2\hat w^\phi$, we obtain 
Applying \eqref{stab:LTPML} to \eqref{eL1}, %\co{and assuming $\max\{k\vep\Linf{\phi}^2,\, k\vep\Linf{\uinc}^2\}\leq\theta_0$} 
we obtain 
\begin{align*} 
k\LtD{\hat w^\phi}+\He{\hat w^\phi}+k^{-1}\sHt{\hat w^\phi}{\OchO} &\leq C_{\mL}\LtD{g} + 3 C_{\mL}k^2 \vep\Linf{\phi+\uinc}^2\LtD{\hat w^\phi} 
%\\
%&\leq C_{\mL}\LtD{g} + 12 C_{\mL}\theta_0 k\LtD{\hat w^\phi}
,
\end{align*}
which leads to the following estimate (by taking  $\theta_0= {1}/{(24 C_\mL)}$ there). 

%The following lemma gives the stability estimate of this auxiliary linearized problem.
\begin{lemma}\label{lem:stabap}
Let the conditions of Lemma \ref{lem:stabl} be satisfied. 
Then there exists a positive constant $\theta_0\ls 1$ such that the solution $\hat w^\phi$ to \eqref{ap1}  
satisfies
\begin{equation}
\label{eq:stabap}
k\LtD{\hat w^\phi}+\He{\hat w^\phi}+k^{-1}\sHt{\hat w^\phi}{\OchO} \leq 2C_\mL\LtD{g},
\end{equation}
under the condition that 
$\max\{k\vep\Linf{\phi}^2,\, k\vep\Linf{\uinc}^2\}\leq\theta_0.$
%
%\eqn{\max\{k\vep\Linf{\phi}^2,\, k\vep\Linf{\uinc}^2\}\leq\theta_0.}
\end{lemma}
%\begin{proof}
%The auxiliary problem \eqref{ap1} %\eqref{ap2} 
%can be rewritten as follows:
%\begin{alignat}{2}
%-\na\cdot(A\na \hat w^\phi)-Bk^2 \hat w^\phi&=g+k^2\vep \oneo\abs{\phi+\uinc}^2\hat w^\phi \quad & \mbox{in }\D;
%\quad \hat w^\phi = 0 \quad  \mbox{on } ~\hat{\Gamma}. \label{eL1}
%%\hat w^\phi&=0\quad & \mbox{on }\hat\Gamma. \label{eL2}
%\end{alignat}
%From Lemma \ref{lem:stabl} and comparing with the linear problem \eqref{LTPML1}, %\eqref{LTPML2}, 
%we arrive at
%\begin{align*} k\LtD{\hat w^\phi}+\He{\hat w^\phi}+k^{-1}\sHt{\hat w^\phi}{\OchO} &\leq C_{\mL}\LtD{g}+C_{\mL}k^2\vep\Linf{\phi+\uinc}^2\LtD{\hat w^\phi} \\
%&\leq C_{\mL}\LtD{g}+2C_{\mL}\theta_0 k\LtD{\hat w^\phi},
%\end{align*}
%which implies the desired estimate \eqref{eq:stabap} by taking  $\theta_0= {1}/{(4C_\mL)}$. 
%\end{proof}
%\begin{remark}\label{rm:existap}
The uniqueness of the solutions to the auxiliary linearized problem \eqref{ap1} follows directly from Lemma \ref{lem:stabap}. Furthermore, the existence of a solution can be obtained by the uniqueness and the Fredholm alternative theorem. In fact, by writing $\hat w^\phi= \hat w^\phi_r + \i \hat w^\phi_i$, where $\hat w^\phi_r$ and $\hat w^\phi_i$ are both real-valued functions, we can get an equivalent variational problem with a bilinear form defined on real-valued Sobolev spaces.
Then it can be proved that the equivalent bilinear form is continuous and satisfies the G{\aa}rding's inequality, 
so an application of the Fredholm alternative theorem will lead to the existence of a solution to the variational problem 
following from the uniqueness; see, e.g., \cite[Theorem 2.34]{Mclean2000} or \cite[\S 6.2]{evans2010}. 
Therefore, the auxiliary linearized problem \eqref{ap1} is well-posed.
%\end{remark}

Moreover, we have the following $L^\infty$-estimate in $\Om_0$ for the solution 
$\hat w^\phi$ to \eqref{ap1}, which will play a crucial role in our subsequent analysis. 
\begin{lemma}\label{lem:linfap}
Let the conditions of Lemma \ref{lem:stabap} be satisfied, then 
the solution $\hat w^\phi$ to \eqref{ap1} satisfies 
\begin{equation}\label{eq:Linfap}
\Linf{\hat w^\phi}\ls k^{\frac{d-3}{2}}\LtD{g}.
\end{equation}
\end{lemma}

\begin{proof}
The estimate \eqref{eq:Linfap} is a direct consequence of the following estimate
\begin{equation}
\label{Linf} 
\Linf{\hat w}\ls k^{\frac{d-3}{2}}\LtD{\hat g},
\end{equation}
where $\hat w\in H_0^1(\D)$ solves the linear problem $\hat L \hat w = \hat g$.
In fact, by rewriting the system \eqref{ap1} %\eqref{ap2} 
to \eqref{eL1}, %\eqref{eL2}, 
then applying the estimate \eqref{Linf} to \eqref{eL1} %\eqref{eL2} 
with $\hat g=g+k^2\vep \oneo \bgS{2\abs{\phi+\uinc}^2\hat w^\phi + (\phi + \uinc)^2 \overline{\hat w^\phi}}$, we obtain  
\begin{align*} 
\Linf{\hat w^\phi} &\ls k^{\frac{d-3}{2}}\bgS{\norm{g}_0+k^2\vep \norm{\phi+\uinc}_{L^\infty(\Omega_0)}^2 \norm{\hat w^\phi}_{0,\Omega_0}}
\ls k^{\frac{d-3}{2}}\big(\LtD{g}+2k\theta_0 \Lt{\hat w^\phi}{\Omz}\big),
\end{align*}
which, together with \eqref{eq:stabap}, gives \eqref{eq:Linfap}. It remains to establish \eqref{Linf}.

Case 1: $d=2$.  We first notice that $0\leq r=|x| \leq R_0:=R-\dist(\Gamma,\Omz)$
for $x\in\Omz$. The solution $\hat w$ in $\Omega_0$ can be solved by separation of variables and expressed 
by Fourier expansions (see \cite[(2.23)--(2.24)]{liwu2019}):
\begin{equation}\label{eq:vwzeta} 
\hat w = v+w+\zeta \,\,\,\text{with}\,\,\,  v=\sum_{n\in\Z} v_n(r) e^{\i n\theta},\, w=\sum_{n\in\Z} w_n(r) e^{\i n\theta}, \,\zeta=\sum_{n\in\Z} \zeta_n(r) e^{\i n\theta},
\end{equation}
%\eqn{\hat w = v+w+\zeta \,\,\,\text{with}\,\,\,  v=\sum_{n\in\Z} v_n(r) e^{\i n\theta},\, w=\sum_{n\in\Z} w_n(r) e^{\i n\theta}, \,\zeta=\sum_{n\in\Z} \zeta_n(r) e^{\i n\theta},
where the coefficients $v_n (r)$, $v_n (r)$ and $\zeta_n(r)$ are given by 
\begin{align*}
v_n (r) &= \frac{\pi\i}{2}\Jn (kr)\int_r^{R} \Hn(kt)g_n(t)t dt + \frac{\pi\i}{2}\Hn (kr)\intr \Jn(kt)g_n(t)t dt, \\
w_n(r) &= \frac{\pi\i}{2}\Jn (kr)\int_R^{\hR} \Hn(k\ttn)g_n(t)t dt, \\
\zeta_n(r) &= \hC_n\Jn(kr),\quad \hC_n = -\frac{\pi\i}{2}\frac{\Hn(k\thR)}{\Jn(k\thR)}\int_0^{\hR}\Jn(k\ttn)g_n(t)tdt,
\end{align*}
where $J_n$ denotes the Bessel function of the first kind with order $n$, and 
$g_n(r)=\frac{1}{2\pi}\int_0^{2\pi} \hat g(r,\theta)e^{-\i n\theta}d\theta$ is the Fourier coefficient of $\hat g$ on $\partial\B_r$. 
%For any domain $G\subset \R^d$, let $\chi_G$ be the characteristic function of $G$. 
It is easy to see that $v$ is the solution to the linear Helmholtz equation $-\Delta v-k^2 v = \hat g \chi_\Om$ with the Sommerfeld radiation condition (i.e. \eqref{eq:Somm}). Thus, 
\[ v(x)=\int_{\R^2} \hat g(y) \chi_\Om (y) G(x,y)dy \quad\text{for }x\in \Omz,
\]
where $G(x,y){=\frac{\i}{4}H_0^{(1)}(k|x-y|)}$ denotes the standard Green's function. From \cite[p. 211]{watson1944}, we have
\eqn{|G(x,y)|\ls \frac1{\sqrt{k|x-y|}} \qaq \int_\Om |G(x,y)|^2{\rm d}y\ls k^{-1},}
%which implies that 
%\eqn{\int_\Om |G(x,y)|^2{\rm d}y\ls k^{-1},}
hence we can easily get 
\[ \Linf{v} \ls k^{-1/2}\Lt{\hat g}{\Om}. \]
Similarly, for $x\in\Omz$, we have
\[ w(x)=\int_{\R^2} \tilde g(y) \chi_{\hOm}(y) G(x,y)dy,
\]
where $\tilde g=\sum_{n\in\N} \tilde g_n(r)e^{\i n\theta}$ and $\tilde g_n(r):=g_n(r)\Hn(k\tr)/\Hn(kr)$. According to \cite[\S 10.21(i)]{olver2010}, the positive zeros of the two real Bessel functions $\Jn(kr)$ and $Y_n(kr)$ are interlaced and hence, $\Hn(kr)=\Jn(kr)+\i Y_n(kr)\not=0$. From Lemma \ref{lem:Bessel}, we have $\abs{\tilde g_n(r)}\leq\abs{g_n(r)}$ and then 
\[ \Linf{w} \ls k^{-1/2}\Lt{\tilde g}{\hOm}\ls k^{-1/2}\Lt{\hat g}{\hOm}.\]
For the last term $\zeta=\sum_{n\in \N}\zeta_n(r)e^{\i n\theta}$ in \eqref{eq:vwzeta},  by applying \cite[(3.24) and (3.33)]{liwu2019}, we get
\[
\abs{\hC_n} \ls e^{-\frac 12 k\siz L}\left(\inthR r\abs{ g_n(r)}^2 dr\right)^{\frac 12}.
\]
Noting that $J_0(kr)^2+2\sum_{n=1}^\infty J_n(kr)^2= 1$ (cf. \cite[(10.23.3)]{olver2010}), we have
\[ \abs{\zeta} =\abs{\sum_{n\in\N} \hC_n\Jn(kr)e^{\i n\theta}}\leq \left(\sum_{n\in\N}\abs{\hC_n}^2\right)^\frac12\left(\sum_{n\in\N}J_n(kr)^2\right)^\frac12 \ls e^{-\frac 12 k\siz L}\LtD{\hat g}.
\]
Since  $\siz L\gs 1$ (see \eqref{eq:assumption}), we have $e^{-\frac 12 k\siz L}\ls k^{-1/2}$, then
\[ \Linf{\zeta} \ls e^{-\frac 12 k\siz L}\LtD{\hat g} \ls k^{-1/2}\LtD{\hat g}. 
\]
Therefore, we have confirmed the validity of \eqref{Linf} for $d=2$.

Case 2: $d=3$. The proof is similar to the case of $d=2$. From \cite[(2.30)--(2.31)]{liwu2019}, $\hat w$ can be expressed by the harmonic expansion
$\hat w = v+w+\zeta$ in $\Omega_0$, 
with
\eqn{v=\sum_{l=0}^{\infty} \sum_{m=-l}^l v_l^m(r) Y_l^m(\theta,\varphi),\; w=\sum_{l=0}^{\infty} \sum_{m=-l}^l w_l^m(r) Y_l^m(\theta,\varphi), \;\zeta=\sum_{l=0}^{\infty} \sum_{m=-l}^l \zeta_l^m(r) Y_l^m(\theta,\varphi),
}
where $Y_l^m$ is the standard spherical harmonics (see \cite[etc.]{watson1944}) and 
\begin{align*}
v_l^m(r) &= \frac{\pi \i}{2} r^{-\frac12}J_{\nu}(kr)\intrR H_\nu^{(1)}(kt)g_l^m(t) t^{\frac32} dt+ \frac{\pi \i}{2} r^{-\frac12}H_{\nu}^{(1)}(kr)\intr J_\nu(kt) g_l^m(t) t^{\frac32} dt,  \\
w_l^m(r) &= \frac{\pi \i}{2} r^{-\frac12}J_{\nu}(kr)\intRhR H_\nu^{(1)}(k\ttn) \be^{-\frac12}g_l^m(t) t^{\frac32} dt   \\
\zeta_l^m(r) &= \hC_\nu r^{-\frac12} J_\nu(kr),  \quad \hC_\nu=-\frac{\pi \i}{2}\frac{H_\nu^{(1)}(k\thR)}{J_\nu(k\thR)} \inthR J_\nu(k\ttn)\be^{-\frac12} g_l^m(t)t^{\frac32} dt, 
\end{align*}
where $\nu := l+\frac12$ and $g_l^m(r)=\int_0^{2\pi}\int_0^{\pi} \hat g(r,\theta,\varphi)Y_l^{-m}(\theta,\varphi)\sin\theta d\theta d\varphi$. Noting that $v$ is the solution to the linear Helmholtz equation $-\Delta v-k^2 v =\hat g\chi_\Om$ with Sommerfeld radiation condition. Thus,
\eqn{v(x)=\int_{\R^3}\hat g(y)\chi_\Om(y) G(x,y)dy,}
where $G(x,y)=\frac{e^{\i k|x-y|}}{4\pi |x-y|}$ denotes the standard Green's function. Obviously, $\abs{G(x,y)} \ls \frac 1{|x-y|}$, then we get 
\eqn{\int_{\Om}\abs{G(x,y)}^2dy \ls 1, \quad\mbox{and then,}\quad \Linf{v} \ls \Lt{\hat g}{\Om}.}
%which implies that
%\eqn{\Linf{v} \ls \Lt{\hat g}{\Om}.}
Similarly, we have
\eqn{w(x)=\int_{\R^3}\tilde g(y)\chi_{\hOm}(y)G(x,y)dy,}
where 
\eqn{\tilde g=\sum_{l=0}^{\infty}\sum_{m=-l}^l \tilde g_l^m(r) Y_l^m(\theta,\varphi) \qaq \tilde g_l^m(r)=\be^{-1/2}g_l^m(r)H_\nu^{(1)}(k\tr)/H_\nu^{(1)}(kr).} 
Using Lemma \ref{lem:Bessel} again, we have $\abs{\tilde g_l^m(r)}\ls \abs{g_l^m(r)}$ and then
\[ \Linf{w}\ls \Lt{\tilde g}{\hOm} \ls \Lt{\hat g}{\hOm}.\]
For the last term $\zeta$, from \cite[(3.40)]{liwu2019} and a similar proof to \cite[(3.24)]{liwu2019}, we have
\eqn{\abs{\hC_\nu} \ls \sqrt{k} e^{-\frac12 k\siz L} \left(\inthR r^2 \abs{g_l^m(r)}^2dr \right)^{\frac12}.}
Hence,
\begin{align*}
\abs{\zeta} &= \abs{\sum_{l=0}^\infty \sum_{m=-l}^l \hC_\nu r^{-\frac12}J_\nu(kr)Y_l^m(\theta,\varphi)} = \abs{ \sum_{l=0}^\infty \sum_{m=-l}^l \hC_\nu \sqrt{\frac{2k}{\pi}} j_l(kr) Y_l^m(\theta,\varphi)}  \\
 &\ls ke^{-\frac12 k\siz L} \LtD{\hat g} \left( \sum_{l=0}^\infty j_l^2(kr)\sum_{m=-l}^l \abs{Y_l^m}^2  \right)^{\frac12},
\end{align*}
where $j_l$ denotes the spherical Bessel function of the first kind and order $l$. From \cite[(2.4.105)]{nedelec2001} and \cite[(10.60.12)]{olver2010}, we have
\[ \sum_{m=-l}^l\abs{Y_l^m}^2 = \frac{2l+1}{4\pi} \qaq \sum_{l=0}^{\infty}(2l+1)j_l^2(kr)=1, 
\]
then we arrive at
\eqn{\Linf{\zeta} \ls ke^{-\frac12 k\siz L}\LtD{\hat g}\ls \LtD{\hat g},}
where we have used $\sigma_0 L \gtrsim 1$. Therefore, \eqref{Linf} also holds for $d=3$. 
%This confirms the desired $L^\infty$-estimate \eqref{Linf}. 
%for the linear problem \eqref{LTPML1}--\eqref{LTPML2}.
%The proof for the case of three dimensions $(d=3)$ is similar, we omit it \cb{to save} space.
%
%Next, we estimate the $L^{\infty}$-estimate for $\hat w^\phi$ of \eqref{ap1}--\eqref{ap2}. From \eqref{eL1}--\eqref{eL2} and \eqref{Linf}, 
%\begin{align*} 
%\Linf{\hat w^\phi} &\ls k^{\frac{d-3}{2}}\big\Vert g+k^2\vep\oneo\abs{\phi+\uinc}^2\hat w^\phi\big\Vert_0 \\
% &\ls k^{\frac{d-3}{2}}\big(\LtD{g}+2k\theta_0 \Lt{\hat w^\phi}{\Omz}\big),
%\end{align*}
%which together with \eqref{eq:stabap} gives \eqref{eq:Linfap}. 
%This completes the proof of the lemma.
\end{proof}

\begin{remark}\label{rm:stabap}
We can easily obtain the estimate from the proof of Lemma \ref{lem:linfap} above that
\begin{equation}\label{eq:Linfap1}
\norm{\hat w^\phi}_{L^{\infty}(\Omega_1)}\ls k^{\frac{d-3}{2}}\LtD{g}
\end{equation}
for any subdomain $\Omega_1$ satisfying $\Omz\subset\subset\Omega_1\subset\subset\Om$ and $\dist(\Gamma,\partial\Omega_1)\eqsim \dist(\partial\Omega_1,\partial\Omz)$.
\end{remark}

\subsection{Existence and stability estimates for the nonlinear PML problem}\label{sec:anas}
In this subsection, we study the well-posedness of the nonlinear PML problem \eqref{eq:PML}. This is carried out by the Newton's iterative process \eqref{ip1}. 
We start with some uniform stability estimates for the solutions $\{\hu^l\}_{l\geq 1}$ 
to \eqref{ip1} in terms of the iteration number $l$, 
with their bounds depending on the constant 
\eq{\label{Mf} \Mf:= \norm{f}_{0,\Omega}+k^2\vep\norm{\uinc}_{L^6(\Om_0)}^3.}
%By applying the Lemmas \ref{lem:stabap} and \ref{lem:linfap}, we may derive the following stability estimates for the sequence $\{\hu^l\}_{l\geq 1}$.
\begin{lemma}\label{lem:pmlip} Under the conditions of Lemma \ref{lem:stabl}, there exists a positive constant $\theta_1\ls 1$ such that if 
\begin{align}
\label{nlcond1}
&k\norm{\hu^0}_{0,\Omega_0}\ls\Mf,\quad \Linf{\hu^0}\ls k^{\frac{d-3}2}\Mf,\\
&\max\big\{k^{d-2}\vep\Mf^2,~k\vep\Linf{\uinc}^2 \big\} \leq\theta_1,\label{nlcond1b}
\end{align}
then the following estimates hold for $l=1,2,\cdots$:
\begin{equation}
k\LtD{\hu^l}+\He{\hu^l}+k^{-1}\sHt{\hu^l}{\OchO} \ls \Mf,\quad \Linf{\hu^l}\ls k^{\frac{d-3}2}\Mf. \label{eq:stabip}
\end{equation}
\end{lemma}
\begin{proof}  
    We set 
    \begin{align*}
    \hat{f}^l := &\;f - k^2 \vep \oneo \bgS{2 |\hu^l + \uinc|^2 \hu^l-(\hu^l+\uinc)^2\overline{\uinc}}\\
     = &\;f - k^2 \vep \oneo \bgS{2 |\hu^l|^2 \hu^l + 2|\hu^l|^2 \uinc + (\hu^l)^2 \overline{\uinc} - |\uinc|^2\uinc}.
    \end{align*}
First, we let $C_\mL$ be the constant from \eqref{eq:stabap} and $C_\infty$ be the hidden generic 
    constant in \eqref{eq:Linfap}. Denote by $\widetilde C_\mL = 4 C_\mL$ and $\widetilde C_\infty = 2 C_\infty$ and let
    $\theta_1 \leq \min \big\{\theta_0, \theta_0/\widetilde C_\infty ^2, 1/(2\widetilde C_\mL \widetilde C_\infty^2 + 3 \widetilde C_\mL \widetilde C_\infty)\big\}$, where $\theta_0$ is from Lemma \ref{lem:stabap}. We also assume that the initial value satisfies $k \|\hu^0\|_{0,\Omega_0}\leq \widetilde C_\mL M(f)$ and $\|{\hu^0}\|_{L^\infty(\Omega_0)}\leq \widetilde C_\infty  k^{\frac{d-3}2}M(f)$. 
    
    Next, we suppose that the following estimates hold for $l=n$ with $n \geq 0$: 
    $$ k\big\Vert \hu^n \big\Vert_{0,\Omega_0} \leq \widetilde{C}_\mathcal{L}M(f) \qaq \big\Vert \hu^n \big\Vert_{L^\infty(\Omz)}\leq\widetilde{C}_\infty k^{\frac{d-3}2}M(f).$$
    Then we can directly get from \eqref{Mf} and \eqref{nlcond1b} that 
    \begin{equation}\label{tfscMf} 
    \begin{aligned}
    \Vert\hat{f}^n \Vert_0 &\leq M(f) + 2 \widetilde C_\mL \widetilde C_\infty^2 k^{d-2}\varepsilon M(f)^3 + 3 \widetilde C_\mL \widetilde C_\infty \theta_1^{\frac12} \varepsilon^{\frac12} k^{\frac{d-2}{2}} M(f)^2 \\
    &\leq \big(1+2 \widetilde C_\mL \widetilde C_\infty^2 \theta_1 + 3 \widetilde C_\mL \widetilde C_\infty \theta_1\big) M(f) \leq 2 M(f),
    \end{aligned}
    \end{equation}
    and $k\varepsilon\Vert\hu^n\Vert_{L^\infty(\Omz)}^2\leq\widetilde{C}_\infty^2k^{d-2}\varepsilon M(f)^2\leq\widetilde{C}_\infty^2\theta_1\leq\theta_0.$
    Therefore, using Lemmas \ref{lem:stabap} and \ref{lem:linfap} we can deduce 
    \[
      k\LtD{\hu^{n+1}}+\He{\hu^{n+1}}+k^{-1}\sHt{\hu^{n+1}}{\OchO} \leq 2 C_\mathcal{L}\big\Vert\hat{f}^{n}\big\Vert_0 \leq \widetilde{C}_\mathcal{L} M(f) \qaq \big\Vert\hu^{n+1}\big\Vert_{L^\infty(\Omega_0)} \leq \widetilde{C}_\infty k^{\frac{d-3}2}M(f).
      \]
Noting that $\widetilde C_\mL$ and $\widetilde C_\infty$ are independent of the iteration number $l$, the proof is completed by induction.

\end{proof}

Now we can establish the well-posedness of the nonlinear PML problem \eqref{eq:PML}. 
%--\eqref{eq:PML:bound:all}. 
%by proving that $\{\hu^l\}_{l\geq 1}$ defined in \eqref{ip1}--\eqref{ip2} is a Cauchy sequence and hence $\hu:=\lim_{l\to\infty}\hu^l$ is a solution to \cb{\eqref{eq:PML}}--\eqref{eq:PML:bound:all}. 
%The proof is similar to that of \cite[Theorem 2.5]{wuzou2018}, so we give only an outline. 
\begin{theorem}\label{thm:stabpml}
Under the conditions of Lemma \ref{lem:stabl}, there exists a positive constant $\theta_2 \ls 1$ such that if  
\begin{equation}
\label{eq:cond}
\max\big\{k^{d-2}\vep\Mf^2,~k\vep\Linf{\uinc}^2 \big\} \leq\theta_2,
\end{equation}
then the nonlinear PML problem \eqref{eq:PML} attains a unique solution $\hu$ satisfying the estimates:
\begin{equation} \label{eq:stab}
k\LtD{\hu}+\He{\hu}+k^{-1}\sHt{\hu}{\OchO} \ls \Mf,\quad \Linf{\hu}\ls k^{\frac{d-3}2}\Mf. 
\end{equation}
%\sout{Moreover, let $\hu^l$ be defined as \eqref{ip1}, then, under the conditions \eqref{nlcond1}, there exists a positive constant $\widetilde{C}_1$ such that if }
%\[\sout{\rho:= \widetilde{C}_1 \cb{\max\big\{k^{d-2}\vep\Mf^2, k\vep\Linf{\uinc}^2\big\}} <1,}\]
%\sout{the following error estimate holds:}
%\eq{\label{eq:hu-hul} \sout{\He{\hu-\hu^l} \ls \rho^l \Mf.}}
\end{theorem}
\begin{proof}
%\cb{Let $\{\hu^l\}_{l\geq 1}$ be defined by \eqref{ip1}.} 
Recalling Newton's sequence $\{\hu^l\}$ from \eqref{ip1},
we can see that the difference $v^l=\hu^{l+1}-\hu^l$ satisfies 
\eqn{ &\;\hat L v^l - k^2 \vep \oneo \BgS{2\abs{\hu^l+\uinc}^2 v^l + \bgS{\hu^l+\uinc}^2 \overline{v^l}} \\
= &\;k^2 \vep \oneo \BgS{\bgS{\hu^l+\uinc}^2 - \bgS{\hu^{l-1}+\uinc}^2}\overline{\big(\hu^{l}+\uinc\big)} - 2k^2\vep\oneo \abs{\hu^{l-1}+\uinc}^2 v^{l-1}.
}
%\begin{alignat*}{2}
%&-\na\cdot(A\na v^l)-Bk^2v^l -k^2\vep \oneo\abs{\hu^{l}+\uinc}^2v^l &\\
%= &~~k^2\vep\oneo\big( \abs{\hu^l+\uinc}^2-\abs{\hu^{l-1}+\uinc}^2\big) \big(\hu^l+\uinc\big) \quad & \mbox{in } \D
%%v^l &= 0 \quad & \mbox{on } \hat{\Gamma}.
%\end{alignat*}
%with the boundary condition $v^l= 0$ on $\hat{\Gamma}$.
By using Lemma \ref{lem:stabap} and noting \eqref{eq:stabip} and \eqref{eq:cond}, we get
\begin{align*}
\He{v^l} &\ls k^2\vep \Big(\Linf{\hu^l+\uinc}^2+\Linf{\hu^{l-1}+\uinc}^2\Big) \Lt{v^{l-1}}{\Om_0} \\ %\Linf{\hu^l+\uinc}\big(\Linf{\hu^l+\uinc}+\Linf{\hu^{l-1}+\uinc}\big)\LtD{v^{l-1}} \\
&\ls \big(k^{d-2}\vep\Mf^2+k\vep\Linf{\uinc}^2\big)k \Lt{v^{l-1}}{\Om_0}. %\\
%&\ls 2\theta_2 \He{v^{l-1}}.
\end{align*}
%which implies that $\He{v^l} \leq \rho k\Lt{v^{l-1}}{\Om_0} \leq \rho \He{v^{l-1}}$ for some constant $\widetilde{C}_1$. 
Then from \eqref{eq:cond}, we let $\theta_2$ be small enough such that 
\[\He{v^l} \leq \tfrac12 k\Lt{v^{l-1}}{\Om_0} \leq \tfrac12 \He{v^{l-1}}, \]
hence we can deduce $\He{v^l} \leq 2^{-l} \He{v^0}$ by induction, which implies that $\{\hu^l\}$ is a Cauchy sequence with respect to the energy norm. Moreover, by using Lemma~\ref{lem:stabap} with $\phi=\hu^l$ and the above estimate again, we  obtain
\[ \sHt{v^l}{\OchO} \ls k\He{v^{l-1}},
\]
which implies that $\{\hu^l\}$ is also a Cauchy sequence with respect to the piecewise $H^2$-norm. Therefore, by taking $l\to \infty$ in \eqref{ip1}, $\hu:=\lim_{l\to\infty}\hu^l$ is a solution to \eqref{eq:PML} %\eqref{eq:PML:bound:all} 
and satisfies the estimates \eqref{eq:stab}.

Suppose that $w$ is another solution to \eqref{eq:PML}, %\eqref{eq:PML:bound:all}, 
with the estimates \eqref{eq:stab}, then $v:=\hu-w$ solves 
\eqn{ &\;\hat L v - k^2 \vep \oneo \bgS{2\abs{w+\uinc}^2 v + (w+\uinc)^2 \overline{v}} \\
= &\;k^2 \vep \oneo \BgS{\bgS{\hu+\uinc}^2 - \bgS{w+\uinc}^2}\overline{\big(\hu+\uinc\big)} - 2k^2\vep\oneo \abs{w+\uinc}^2 v.
}
%\begin{alignat*}{2}
%&-\na\cdot(A\na v)-Bk^2v -k^2\vep \oneo\abs{w+\uinc}^2 v \\
%= &~k^2\vep\oneo\big( \abs{\hu +\uinc}^2 -\abs{w +\uinc}^2\big) \big(\hu +\uinc\big) \quad & \mbox{in } \D 
%%v &= 0 \quad & \mbox{on } \hat{\Gamma}
%\end{alignat*}
%with the boundary condition $v= 0$ on  $\hat{\Gamma}$. 
Similarly to the above derivations, if $\theta_2$ is small enough, we can prove that $\He{v} \leq 2^{-1}\He{v}$, which implies that $v=0$. 
%Moreover, we have 
%\[ \He{v^l} \leq \rho^{l-1}\He{v^1} \leq \rho^l k\Lt{v^0}{\Om_0} \leq \rho^l\Mf \qaq
%\He{\hu-\hu^l} \leq \frac{\rho^l}{1-\rho}\Mf.\]
This completes the proof of Theorem \ref{thm:stabpml}.
\end{proof}
\begin{remark}\label{rmk:low-energy}
%{\rm (i)} 
Theorem \ref{thm:stabpml} says that the nonlinear PML problem attains a unique solution 
with low energy, i.e., among all the solutions with energy below an upper bound (as specified by the stability estimates 
in \eqref{eq:stab}), under the condition \eqref{eq:cond} indicating that the incident wave and the nonlinearity may not be too strong.
But this result does not exclude the possibility of multiple solutions to the nonlinear PML problem, nor does it cover the case of strong nonlinearity. 

%\sout{We see from Theorem \ref{thm:stabpml} that the nonlinear PML system \eqref{eq:PML} is well-posed under the crucial condition \eqref{eq:cond}. But we know the NLH system \eqref{eq:Helm}--\eqref{eq:Somm} have actually multiple physical solutions, 
%as we shall see numerically in Section \ref{s:5}. Therefore we can learn from Theorem \ref{thm:stabpml} that those solutions with a finite energy $\Mf$ up to a constant (see \eqref{eq:stab}) are unique and stable, 
%but there are multiple physical solutions if we consider larger energies. }

%\sout{ {\rm (ii)} It is known that the frozen-nonlinearity iteration \eqref{eq:fp} converges only linearly to the low-energy solution for problems with weak nonlinearity. By contrast, the Newton's iteration converges quadratically if the initial guess is close to an exact solution even for problems with strong nonlinearity (see  Theorem \ref{thm:quad}).} %below for details).
%\sout{From \eqref{eq:hu-hul} and the proof of Theorem \ref{thm:stabpml}, the sequence $\{\hu^l\}_{l\geq 1}$ defined by \eqref{ip1} actually converges linearly to the solution with the least energy, 
%if the initial values is small enough, in other words, close to $0$. The analyses for the two other methods \eqref{eq:fp} and \eqref{eq:mN} are similar. In general, the sequence \eqref{eq:fp} cannot converge to the other solutions with larger energies, but Newton's sequence \eqref{ip1} (also defined in \eqref{ip1}) can, as long as the initial value is close enough to those solutions, see  Theorem \ref{thm:quad} for details.}
\end{remark}

Before proceeding, we recall the following the Nirenberg inequality \cite{nirenberg1959}:
\eqn{
\norm{v}_{L^p(\Om)}\le C_{\mathrm{Nir},p}\norm{v}_{1,\Om}^{\frac{d}{2}-\frac dp}\norm{v}_{0,\Om}^{1-\frac{d}{2}+\frac dp}\quad\forall\, v\in H^1(\Om),\; 2\leq p \leq \tfrac{2d}{d-2}.}
Denote by $\cni=\max\{(1+\sigma_0^2)^{\frac d8}C_{\mathrm{Nir},4}, (1+\sigma_0^2)^{\frac d6}C_{\mathrm{Nir},6}\}$. By noting \eqref{h12d}--\eqref{h13d} and $k\gg 1$, we have
\eq{\label{ni}
\norm{v}_{L^4(\Om)} \leq \cni k^{\frac d4 -1}\He{v}\qaq \norm{v}_{L^6(\Om)}\le {\cni}k^{\frac{d}{3}-1}\He{v}.}

The following theorem gives a quadratic convergence result for the Newton's iteration \eqref{ip1}.
\begin{theorem}\label{thm:quad}
  %Under the conditions of Lemma \ref{lem:pmlip} and Theorem \ref{thm:stabpml}, 
  Let $\hu^*$ be one of the multiple solutions to the nonlinear PML problem \eqref{eq:PML} and  $w^* \in H_0^1(\D)$ be the solution to the linearized PML problem
  \eq{\label{eq:ws}\hat L w^* - k^2 \vep \oneo \bgS{2\abs{\hu^* + \uinc}^2 w^* + (\hu^* + \uinc)^2 \overline{w^*}} = g^*.}
Suppose the following stability estimate holds for any given function $g^*\in L^2(\D)$,
  \eq{\label{eq:stab:ws} \He{w^*} \leq C^*\norm{g^*}_0}
 where $C^*$ may depend on $k$ and $\hu^*$.   Denote  by
  \eqn{N^* := 15C^ *\cni^2  k^{\frac{d}{2}}\vep\Linf{\hu^*+\uinc} \qaq \ta^* := \min \Ap{\big(N^*\big)^{-1}, \cni^{-1}k^{1-\frac{d}{2}}\Linf{\hu^*+\uinc}}. }
  If the initial guess $\hu^0$ satisfies $\He{\hu^0-\hu^*}\le\ta^*$,
   then the Newton's iterative sequence $\{\hu^l\}_{l\geq 1}$ defined by \eqref{ip1} converges quadratically to $\hu^*$, namely,
  \begin{equation}\label{eq:quad}
  \He{\hu^{l+1} - \hu^*} \ls N^* \He{\hu^l - \hu^*}^2, \quad l =0, 1, 2, \cdots.
  \end{equation}
  \end{theorem}
\begin{proof}
Denote $e^l = \hu^l -\hu^*$. From \eqref{ip1} and \eqref{eq:PML}, it is easy to verify that
%\eqn{%&\; 
%\hat L e^{l+1} - k^2 \vep \oneo\BgS{2\abs{\hu^l+\uinc}^2e^{l+1} +  \bgS{\hu^l+\uinc}^2\overline{e^{l+1}}} 
%%=&\; k^2 \vep \oneo \BgS{-2\abs{\hu^l+\uinc}^2 e^l - \abs{\hu+\uinc}^2\bgS{\hu+\uinc} + \bgS{\hu^l+\uinc}^2 \overline{\bgS{\hu+\uinc}}} \\
%%=&\; -k^2\vep\oneo \BgS{2\abs{e^l}^2\bgS{\hu+\uinc} + \bgS{e^l}^2\overline{\bgS{\hu+\uinc}} + 2\abs{e^l}^2 e^l} \\
%=%&\; 
%-k^2\vep\oneo \BgS{2\abs{e^l}^2\bgS{\hu+\uinc} + \bgS{e^l}^2\overline{\bgS{\hu^l+\uinc}}}.
%}
\eq{\label{NT1}&\; \hat L e^{l+1} - k^2 \vep \oneo\BgS{2\abs{\hu^*+\uinc+e^l}^2e^{l+1} +  \bgS{\hu^*+\uinc+e^l}^2\overline{e^{l+1}}}\notag \\
=&\; -k^2\vep\oneo \BgS{2\abs{e^l}^2\bgS{\hu^*+\uinc+e^l} + \bgS{e^l}^2\overline{\bgS{\hu^*+\uinc}}},
}
which can be rewritten as
\eqn{&\; \hat L e^{l+1} - k^2 \vep \oneo\BgS{2\abs{\hu^*+\uinc}^2e^{l+1} +  \bgS{\hu^*+\uinc}^2\overline{e^{l+1}}} \\
=&\; -k^2\vep\oneo \BgS{2\abs{e^l}^2\bgS{\hu^*+\uinc+e^l} + \bgS{e^l}^2\overline{\bgS{\hu^*+\uinc}}} \\
&\; +k^2\vep\oneo \BgS{2\abs{e^l}^2 e^{l+1} + 4e^{l+1}\Re\bgS{\overline{e^l}(\hu^*+\uinc)} + (e^l)^2\overline{e^{l+1}} + 2e^l(\hu^*+\uinc) \overline{e^{l+1}}}.
}
Applying  \eqref{eq:stab:ws} and \eqref{ni}, %\eqref{eq:stabip}, \eqref{eq:stab} and Lemma \ref{lem:stabap} and noting $H^1 \hookrightarrow L^6$, 
we have
%\eqn{\He{e^{l+1}} &\ls k^2\vep \norm{e^l}_{H^1(\Om_0)}^2 \bgS{\norm{\hu}_{L^6(\Om_0)} + \norm{\hu^l}_{L^6(\Om_0)} + \norm{\uinc}_{L^6(\Om_0)}} \\
%&\ls k^2\vep \He{e^l}^2 \BgS{\bgS{k^{-1}}^\frac13 \bgS{k^{\frac{d-3}2}}^\frac23 \Mf+ \norm{\uinc}_{L^6(\Om_0)}} \\
%&\ls \bgS{k^{\frac{d+2}3}\vep \Mf + k^2\vep\norm{\uinc}_{L^6(\Om_0)}} \He{e^l}^2,
%}
%where we used $\norm{\cdot}_{L^6} \ls \norm{\cdot}_{L^2}^{\frac13} \norm{\cdot}_{L^\infty}^{\frac23}$ in the second inequality. This allow us to conclude the proof.
\eqn{\He{e^{l+1}} \leq&\; C^* k^2 \vep \Big(3\Linf{\hu^*+\uinc} \norm{e^l}_{L^4(\Om_0)}^2 +  2\norm{e^l}_{L^6(\Om_0)}^3  \\
&\; + 6\Linf{\hu^*+\uinc}\norm{e^l}_{L^4(\Om_0)}\norm{e^{l+1}}_{L^4(\Om_0)}+ 3\norm{e^l}_{L^6(\Om_0)}^2 \norm{e^{l+1}}_{L^6(\Om_0)} \Big)\\
\leq&\; C^*  \cni^2  \big(6k^{\frac{d}{2}}\vep\Linf{\hu^*+\uinc}+4 \cni k^{d-1}\vep\He{e^l}\big) \big(\tfrac12\He{e^l}^2+\He{e^l}\He{e^{l+1}}\big).}
Supposing $\He{e^l} \leq \ta^* $, we get
\eqn{\He{e^{l+1}} &\leq 10C^*\cni^2  k^{\frac{d}{2}}\vep\Linf{\hu^*+\uinc} \big(\tfrac12\He{e^l}^2+\He{e^l}\He{e^{l+1}}\big)\\
&\leq\tfrac13 N^*\He{e^l}^2+\tfrac23 N^*\ta^*\He{e^{l+1}}\le\tfrac13N^*\He{e^l}^2+\tfrac23\He{e^{l+1}},}
which yields $\He{e^{l+1}} \leq N^* \He{e^l}^2$ and moreover $\He{e^{l+1}} \leq N^*(\ta^*)^2 \le \ta^*$. Then the proof of the theorem follows by induction. \end{proof}

\begin{remark}\label{rmk:quad-low}
%{\rm (i)} \sout{Compared with \eqref{eq:fp}, the Newton's method \eqref{ip1} is quadratic convergent and more robust, that is, as long as the initial values $\hu^0$ is close enough to $\hu^*$, the Newton's iterative sequence given by \eqref{ip1} converges quadratically to $\hu^*$, even if $\hu^*$ is with larger energy.}

{\rm (i)} From the proof of Theorem~\ref{thm:stabpml}, we can see that, under the conditions of Theorem~\ref{thm:stabpml}, the low-energy solution  $\hu^*=\hu$ satisfies the stability condition \eqref{eq:stab:ws} with $C^*\eqsim 1$. Therefore, $\{\hu^l\}_{l\geq 1}$, defined by \eqref{ip1}, converges quadratically to $\hu^*$, as long as $\hu^0$ is close enough to $\hu^*$. 

{\rm (ii)} If $\hu^*$ is with large but finite energy, the stability estimate \eqref{eq:stab:ws} should also be satisfied but we fail to prove it.  But by using the Fredholm alternative theorem \cite{gilbarg2001}, we can at least claim that the estimate \eqref{eq:stab:ws} holds for every real $k$ except possibly for a discrete set of values, 
and in this case, $\{\hu^l\}$ also converges quadratically to $\hu^*$ if the initial guess is sufficiently close to the exact solution.
\end{remark}

\subsection{Convergence of the nonlinear PML solution}
In this subsection, 
we prove that the solution to the nonlinear PML problem \eqref{eq:PML} %--\eqref{eq:PML:bound:all} 
converges to the solution to the NLH \eqref{eq:Helm}--\eqref{eq:Somm} exponentially in the domain 
$\Omega$, in terms of both the PML medium parameter $\sigma_0$ and thickness $L$. 

\subsubsection{A linear auxiliary problem for NLH}
In this subsection, we follow the analysis of the nonlinear PML problem in Section \ref{au-problem} to first introduce an auxiliary problem for NLH \eqref{eq:Helm}--\eqref{eq:Somm}. %and its PML approximation. 
Then we give some stability results for the auxiliary problem 
and an exponential convergence estimate for its PML approximation. 
%\cb{\sout{As these estimates are quite similar to those of \eqref{ap1} and the proofs in} \cite{melenk2010,liwu2019}, \sout{we omit their detailed proofs here.}}

Denote by $L$ the linear Helmholtz operator, i.e., $Lw=-\Delta w - k^2 w$. 
%\sout{For a given source $g\in L^2(\R^d)$ \cb{with compact support, let $w\in H^2(\R^d)$ be the solution to}  the linear Helmholtz equation with Sommerfeld \cb{radition} condition, namely,   }
%\eq{\label{eq:LH}\sout{\cb{L w} = g \quad \text{in }\R^d; \quad\, \abs{\partial_r w - \i k w} = o(r^{\frac{1-d}2}) \quad \text{for } r\to \infty.}}
Similarly to what we did in Section \ref{au-problem}, we start with an auxiliary linearized problem 
of the original NLH \eqref{eq:Helm}--\eqref{eq:Somm} through Newton's iteration:
for given $\phi\in L^\infty(\R^d)$ and $g\in L^2(\R^d)$ supported in $\Om$, find $w^\phi\in H_{\rm loc}^2(\R^d)$ such that
\eq{\label{eq:Helmap} L w^\phi - k^2\vep\oneo \bgS{2\abs{\phi+\uinc}^2 w^\phi + (\phi+\uinc)^2 \overline{w^\phi}} = g; \quad \abs{\partial_r w^\phi - \i k w^\phi} = o(r^{\frac{1-d}2}) \quad \text{as } r\to \infty.
}
%\begin{alignat}{2}
%-\Delta u^\phi - k^2u^\phi - k^2\vep\oneo\abs{\phi+\uinc}^2(u^\phi+\uinc) &= f &\quad & \text{in } \R^{d}, \label{eq:Helmap}\\
%\abs{\frac{\partial u^\phi}{\partial r} - \i ku^\phi} &= o(r^{\frac{1-d}2}) &\quad &\text{for } r\to \infty. \label{eq:Sommap}
%\end{alignat}
By using the same techniques as those used in Lemmas \ref{lem:stabap}, \ref{lem:linfap} and the existing stability estimates for the linear Helmholtz equation (see e.g. \cite[Lemma 3.5]{melenk2010}), we can obtain the following results. The details are omitted. 
 
\begin{lemma}\label{lem:oap} Suppose $\supp g\subset\Om$. There exists a positive constant $\tilde\theta_0 \ls 1$ such that if
\eqn{ \max\{k\vep\Linf{\phi}^2,k\vep\Linf{\uinc}^2\}\leq\tilde\theta_0, 
}
then the solution $w^\phi$ to \eqref{eq:Helmap} satisfies
\[ k\Lt{w^\phi}{\Om} +\Ho{w^\phi}{\Om}+k^{-1}\Ht{w^\phi}{\Om} \ls \Lt{g}{\Om}\qaq \Linf{w^\phi}\ls k^{\frac{d-3}2}\Lt{g}{\Om}.
\]
\end{lemma}
Similarly to the discussion below Lemma \ref{lem:stabap}, the well-posedness of the auxiliary problem \eqref{eq:Helmap} follows from the Fredholm alternative theorem and the above stability estimate. We omit the details here.

Recalling the definition \eqref{ap1}, we see that $\hat w^\phi$ is the PML approximation of $w^\phi$ 
to \eqref{eq:Helmap}. In fact, we have  the following exponentially convergence result for $\hat w^\phi$. 
\begin{lemma}\label{lem:convap} Suppose $\supp g\subset\Om$, and the conditions of Lemma \ref{lem:stabl} 
are satisfied. 
There exists a positive constant $\hat\theta \ls 1$ such that if
\eqn{ \max\{k\vep\Linf{\phi}^2,k\vep\Linf{\uinc}^2\}\leq \hat\theta, 
} then
\eq{\label{eq:w-phi}
\He{w^\phi-\hat w^\phi}_\Om \ls k^5 e^{-2 k\siz L\big(1-\frac{R^2}{\hR^2+\siz^2L^2}\big)^{\frac12}}\Lt{g}{\Om}.
}
\end{lemma}
\begin{proof}
We first see that $\hat w^\phi$ and $w^\phi$ satisfy 
$L w^\phi = g^\phi$ and $\hat L \hat w^\phi = \hat g^\phi$, respectively, 
with $g^\phi$ and $\hat g^\phi$ given by 
\eqn{g^\phi &= g + k^2\vep\oneo \bgS{2\abs{\phi+\uinc}^2 w^\phi + (\phi+\uinc)^2 \overline{w^\phi}}, \\
\hat g^\phi &= g + k^2\vep\oneo \bgS{2\abs{\phi+\uinc}^2 \hat w^\phi + (\phi+\uinc)^2 \overline{\hat w^\phi}},}
%thus we have $L w^\phi = g^\phi$ and $\hat L \hat w^\phi = \hat g^\phi$. 
Let $e^\phi \in H^1_0(\D)$ be the solution to $\hat L e^\phi = g^\phi - \hat g^\phi$. 
From Lemma~\ref{lem:stabl}, we get
\begin{equation}\label{eq:e-phi} 
\begin{aligned}
\He{e^\phi} &\ls \norm{g^\phi - \hat g^\phi}_{0} \ls k^2 \vep \bgS{\Linf{\phi}^2+\Linf{\uinc}^2} \norm{w^\phi - \hat w^\phi}_{0,\Om_0} 
\ls \hat\theta \He{w^\phi - \hat w^\phi}_{\Om_0}.
\end{aligned}
\end{equation}

On the other hand, it is easy to see that $\hat L(\hat w^\phi + e^\phi) = g^\phi$, which implies that $(\hat w^\phi + e^\phi)$ is the PML approximation of the linear Helmholtz problem $L w^\phi = g^\phi$. Let $\hat\theta \le \tilde\theta_0$ so that Lemma~\ref{lem:oap} holds. Applying the existing convergence result in \cite[Theorem 3.7]{liwu2019} and the trace inequality, we have
\eqn{\He{w^\phi - (\hat w^\phi + e^\phi)}_\Om \ls k^5 e^{-2 k\siz L\big(1-\frac{R^2}{\hR^2+\siz^2L^2}\big)^{\frac12}} \norm{w^\phi}_{H^{\frac12}(\partial\Om)} 
\ls  k^5 e^{-2 k\siz L\big(1-\frac{R^2}{\hR^2+\siz^2L^2}\big)^{\frac12}} \norm{g}_{0,\Om},
}
which together with \eqref{eq:e-phi} gives \eqref{eq:w-phi}, as long as $\hat\theta$ is small enough.
\end{proof}

\subsubsection{Stability estimates of the solutions to NLH}
Similarly to what we did for the nonlinear PML problem \eqref{eq:PML}, %\eqref{eq:PML:bound:all}, 
we study the well-posedness of the NLH \eqref{eq:Helm}--\eqref{eq:Somm}, 
by an Newton's iterative process: %constructed by a linearization of the nonlinear system: 
For a given $u^0\in H^2(\R^d)$ satisfying the radiation condition \eqref{eq:Somm}, 
find $u^{l+1}\in H^2(\R^d)$ satisfying the condition \eqref{eq:Somm} for $l=0,1,2,\cdots$, such that 
\begin{equation}\label{eq:Helmip}
\begin{aligned}
 L u^{l+1} &- k^2 \vep \oneo \BgS{2\abs{u^{l}+\uinc}^2 u^{l+1} + \big(u^{l}+\uinc\big)^2 \overline{u^{l+1}}} \\
= f &- k^2 \vep \oneo \BgS{2\abs{u^{l} + \uinc}^2 u^{l}-\big(u^{l}+\uinc\big)^2\overline{\uinc}}.
\end{aligned}
\end{equation}
%\begin{alignat}{2}
%-\Delta u^l - k^2u^l -k^2\vep\oneo\abs{u^{l-1}+\uinc}^2 (u^l+\uinc) &= f &\quad & \text{in } \R^{d}, \label{eq:Helmip}\\
%\abs{\frac{\partial u^l}{\partial r} - \i ku^l} &= o(r^{\frac{1-d}2}) &\quad &\text{for } r\to \infty. \label{eq:Sommip}
%\end{alignat}
By following the proofs of Lemma~\ref{lem:pmlip} and Theorem~\ref{thm:stabpml} and applying Lemma \ref{lem:oap}, we can obtain the stability estimates of the sequence $\{u^l\}_{l\geq 1}$ and the solution $u$ to the NLH \eqref{eq:Helm}--\eqref{eq:Somm} as stated below. %The proofs are omitted.
\begin{lemma}\label{lem:oip}
There exists a positive constant $\tilde\theta_1 \ls 1$ such that if 
\begin{align*}
&k\norm{u^0}_{0,\Omega_0}\ls\Mf,\quad \Linf{u^0}\ls k^{\frac{d-3}2}\Mf,\quad 
\max\big\{k^{d-2}\vep\Mf^2,~k\vep\Linf{\uinc}^2 \big\} \leq\tilde\theta_1,
\end{align*}
then the following estimates hold for $l=1,2,\cdots$:
\[
\He{u^l}_\Om+k^{-1}\Ht{u^l}{\Om} \ls \Mf,\qaq \Linf{u^l}\ls k^{\frac{d-3}2}\Mf. 
\]
%where $\cb{\tilde f}=f+k^2\vep\oneo\abs{\uinc}^2\uinc=f+\Delta\uinc+k^2(1+\vep\oneo\abs{\uinc}^2)\uinc$. 
Moreover, there exists a positive constant $\tilde\theta_2 \ls 1$ such that if 
\begin{equation} 
\label{eq:cond2}
\max\big\{k^{d-2}\vep\Mf^2,~k\vep\Linf{\uinc}^2 \big\} \leq\tilde\theta_2,
\end{equation}
then the NLH system \eqref{eq:Helm}--\eqref{eq:Somm} attains a unique solution $u=\lim_{l\to\infty}u^l$ satisfying:
\begin{equation}\label{eq:NLH_well-posedness}
\He{u}_{\Om}+k^{-1}\Ht{u}{\Om}+k^{\frac{3-d}2}\Linf{u} \ls \Mf.
\end{equation}
\end{lemma}

\begin{remark} 
Similarly to Theorem \ref{thm:quad}, we can prove a quadratic convergence result for \eqref{eq:Helmip}. The details are omitted.
\end{remark}

\subsubsection{Convergence estimates}
Now we turn to the approximation error estimates between the nonlinear PML problem \eqref{eq:PML} and its original NLH \eqref{eq:Helm}--\eqref{eq:Somm}. 
\begin{theorem}\label{thm:conver}
Let $u$ and $\hu$ be the solutions to \eqref{eq:Helm}--\eqref{eq:Somm} and \eqref{eq:PML}, 
%\eqref{eq:PML:bound:all}, 
respectively. Then under the conditions of Lemma \ref{lem:stabl}, 
there exists a positive constant $\tilde\theta\ls 1$ such that 
the following estimate holds
\begin{equation} \label{eq:con}
\He{u-\hu}_\Om \ls k^5 e^{-2 k\siz L\big(1-\frac{R^2}{\hR^2+\siz^2L^2}\big)^{1/2}}\Mf, 
\end{equation}
if 
\begin{equation}
\label{eq:cond3}
\max\big\{k^{d-2}\vep\Mf^2,~k\vep\Linf{\uinc}^2 \big\} \leq\tilde\theta\,.
\end{equation}
\end{theorem}
\begin{proof} Suppose that $\tilde\theta\leq\min \{\theta_2,\tilde\theta_2\}$ where $\theta_2$ and $\tilde\theta_2$  are from Theorem~\ref{thm:stabpml} and Lemma~\ref{lem:oip}, respectively.
For simplicity, we denote by
\[ \EPML:=k^5 e^{-2 k\siz L\big(1-\frac{R^2}{\hR^2+\siz^2L^2}\big)^{1/2}}.
\]
Since $u$ and $\hu$ are the limits of $\{u^l\}_{l\ge 1}$ in Lemma \ref{lem:oip} and $\{\hu^l\}_{l\ge 1}$ in Lemma \ref{lem:pmlip}, respectively, it suffices to estimate the error $u^l-\hu^l$. 

Define $\cu^0=\hu^0$ and let $\cu^{l+1}\in H_0^1(\D)$ for $l=0,1,2,\cdots$ solve 
\eqn{ \hat L\cu^{l+1} &- k^2 \vep \oneo \BgS{2\abs{u^{l}+\uinc}^2 \cu^{l+1} + \big(u^{l}+\uinc\big)^2 \overline{\cu^{l+1}}} \\
= f &- k^2 \vep \oneo \BgS{2\abs{u^{l} + \uinc}^2 u^{l}-\big(u^{l}+\uinc\big)^2\overline{\uinc}}.
}
%\begin{alignat*}{2}
%-\na\cdot(A\na \cb{\cu}^l)-Bk^2\cb{\cu}^l-k^2\vep \oneo\abs{u^{l-1}+\uinc}^2(\cb{\cu}^l+\uinc) &= f \quad & \mbox{in } \D, 
%% \cb{\cu}^l &= 0 \quad & \mbox{on } \hat{\Gamma}.
%\end{alignat*}
%with the boundary condition $\cb{\cu}^l = 0$ on $\hat{\Gamma}$.
Clearly, $u^l-\hu^l=(u^l-\cu^l)+(\cu^l-\hu^l)$. From  Lemma \ref{lem:convap} with $\phi=u^l$, Lemma \ref{lem:oip}, and \eqref{eq:cond3}, and following the procedure in \eqref{tfscMf}, we conclude that
\begin{equation}\label{eq:pf1}
\He{u^{l+1}-\cu^{l+1}}_\Om \ls \EPML \norm{f^l}_{0,\Om} %\big\Vert f+k^2\vep\oneo\abs{u^{l-1}+\uinc}^2\uinc \big\Vert_{0,\Om} 
\ls  \EPML \Mf, \quad l\geq 0,
\end{equation}
where $f^l :=f - k^2 \vep \oneo \big( 2|u^{l} + \uinc|^2 u^{l}-(u^{l}+\uinc)^2\overline{\uinc} \big)$.
%\eqn{f^l :=&\;f - k^2 \vep \oneo \BgS{2\abs{u^{l} + \uinc}^2 u^{l}-\big(u^{l}+\uinc\big)^2\overline{\uinc}} \\
%=&\;f - k^2 \vep \oneo \BgS{2\abs{u^l}^2 u^l + 2\abs{u^l}^2 \uinc + \bgS{u^l}^2 \overline{\uinc} - \abs{\uinc}^2\uinc}.
%}
We still need to estimate $\eta^l:=\cu^l-\hu^l$. It is easy to verify that the sequence $\{\eta^l\}_{l\geq 0}$ satisfies the following recursive relation: 
\eqn{&\; \hat L \eta^{l+1} - k^2 \vep \oneo \BgS{2\abs{u^l+\uinc}^2 \eta^{l+1} + \big(u^l+\uinc\big)^2 \overline{\eta^{l+1}}} \\
=&\; k^2\vep\oneo\Big( 2\bgS{\abs{u^l+\uinc}^2-\abs{\hu^l+\uinc}^2}\bgS{\hu^{l+1}- u^l} - 2\abs{\hu^l+\uinc}^2\bgS{u^l - \hu^l} \\
&\; +\bgS{\bgS{u^l+\uinc}^2 - \bgS{\hu^l+\uinc}^2} \overline{\bgS{\hu^{l+1} + \uinc}} \Big).
}
%\begin{alignat*}{2}
%-\na&\cdot(A\na \eta^l)-Bk^2\eta^l-k^2\vep \oneo\abs{u^{l-1}+\uinc}^2\eta^l &\\
%&= k^2\vep\oneo\big\{\abs{u^{l-1}+\uinc}^2 
%-\abs{\hu^{l-1}+\uinc}^2\big\}\big(\hu^l+\uinc\big) \quad & \mbox{in } \D, 
%% \eta^l &= 0 \quad & \mbox{on } \hat{\Gamma}.
%\end{alignat*}
%with the boundary condition $\eta^l = 0$ on $\hat{\Gamma}$.
Then we can %follow the proof of Theorem~\ref{thm:stabpml}, 
use Lemmas \ref{lem:stabap}, \ref{lem:pmlip} and \ref{lem:oip} to get 
\begin{align*}
\He{\eta^{l+1}}_{\Om}\ls\He{\eta^{l+1}} &\ls k^2\vep \Sp{\Linf{\uinc}^2 + \Linf{u^l}^2 +\Linf{\hu^l}^2 + \Linf{\hu^{l+1}}^2} \norm{u^l-\hu^l}_{0,\Om_0} \\
%\Linf{\hu^l+\uinc}\big\{\Linf{u^{l-1}+\uinc} \\
%&~\quad\quad+\Linf{\hu^{l-1}+\uinc} \big\}\Lt{u^{l-1}-\hu^{l-1}}{\Omz} \\
% &\ls \big(k^{d-2}\vep\Mf^2+\tilde\theta\big)\big(\He{u^{l-1}-\cb{\cu}^{l-1}}_\Om +\He{\eta^{l-1}}_\Om\big) \\
 &\ls  \tilde\theta \big(\He{u^l-\cu^l}_\Om +\He{\eta^l}_\Om\big).
\end{align*}
Now letting $\tilde\theta$ be sufficiently small such that
\[ \He{\eta^{l+1}}_\Om \leq \tfrac12\He{u^l-\cu^l}_\Om +\tfrac12\He{\eta^l}_\Om, 
\]
then by induction, using \eqref{eq:pf1} and noting that $\eta^0=0$, we get 
\[ \He{\eta^l}_\Om \ls \sum_{j=0}^{l-1} 2^{j-l}\He{u^{j}-\cu^{j}}_\Om \ls \EPML\Mf +2^{-l}\He{u^{0}-\cu^{0}}_\Om,
\]
which together with \eqref{eq:pf1} implies that
\[ \He{u^l-\hu^l}_\Om \ls \EPML\Mf +2^{-l}\He{u^{0}-\hu^{0}}_\Om.
\]
Taking $l\to\infty$ allows us to conclude the proof.
\end{proof}

\section{FEM and its error estimates}\label{s:4}
In this section, we introduce the FEM for the nonlinear PML problem \eqref{eq:PML} %\eqref{eq:PML:bound:all} 
and prove the stability and preasymptotic error estimates for the finite element (FE) solution.

\subsection{FEM and the elliptic projection}
Let $\Th$ be a curvilinear triangulation of $\D$. For any $K\in\Th$, we define $h_K:=\diam(K)$ and $h:=\max_{K\in \Th} h_K$. Assume that $h_K\eqsim h$ %, and $\mathring{K}\cap\Gamma =\emptyset$ 
for any $K \in\Th$. Additionally, we denote by $\widehat{K}$ the reference element and $F_K$ the element map from $\widehat{K}$ to $K\in \Th$ (see \cite[Assumption~5.2]{melenk2010}).
%which satisfies Assumption~5.2 in \cite{melenk2010}, that is, $F_K$ can be written as $F_K=R_K\circ A_K$, where $A_K$ is an affine map and $R_K$ is an analytic map. Moreover, there are positive 
%constants $C_{\text{affine}},~C_{\text{metric}}$ and $\bar{\gamma}>0$ independent of $h$ such that 
%the maps $R_K$ and $A_K$ satisfy (with $\widetilde{K}=A_K(\widehat{K})$)
%\begin{align*}
%&\Vert A'_K\Vert_{L^\infty(\widehat K)}\leq C_{\text{affine}}h,\quad \Vert (A'_K)^{-1}\Vert_{L^\infty(\widehat K)}\leq C_{\text{affine}}h^{-1}, \\
%&\Vert (R'_K)^{-1}\Vert_{L^\infty(\widetilde K)}\leq C_{\text{metric}},\quad \Vert \na^n R_K\Vert_{L^\infty(\widetilde K)}\leq C_{\text{affine}}\bar{\gamma}^n n!\quad \forall n\in\N_0.
%\end{align*}
%In fact,  $R_K\in C^2(\widetilde K)$ is enough for the case of linear FEM. 
For simplicity, we assume that the triangulation $\Th$ fits the interfaces $\pa\Om_0$ and $\Gamma$, that is, $\pa\Om_0$ and $\Gamma$ do not pass through the interior of any element $K\in\Th$. 

Let $V_h$ be the linear finite element approximation space 
\begin{equation*}
V_h:=\{v_h\in H_0^1(\D):v_h|_K\circ F_K \in \mathcal{P}_1(\widehat{K})\;\;\;\forall K\in\Th\},
\end{equation*}
where $\mathcal{P}_1(\widehat{K})$ denotes the set of all first order polynomials on $\widehat{K}$. 
Recalling $\anl$ defined in \eqref{a1}, 
then the FEM for the nonlinear PML problem \eqref{varnltpml} reads as: find $u_h \in V_h$ such that
\begin{equation}\label{eq:FEM}
\anl(u_h,v_h) = \inOm{f}{v_h} \quad\forall v_h\in V_h.
\end{equation}

For further analysis, we let $I_h$ denote the standard finite element  interpolation operator onto $V_h$ (see, e.g., \cite[\S 3.3]{brenner2008}).
Moreover, we shall need two elliptic projections $P_h,\,P_h^* : H_0^1(\D)\mapsto V_h$ defined by
\begin{equation}
(A\na v_h,\na P_hw)=(A\na v_h,\na w),\quad (A\na P_h^* w, \na v_h)=(A\na w,\na v_h) \quad \forall v_h\in V_h,\, w\in H_0^1(\D). \label{eq:Ph}
\end{equation}
Noting that $A$ is symmetric, it is easy to verify that $P_h^*w=\overline{P_h\overline{w}}$.
%\begin{equation}\label{eq:Ph2} 
%(A\na w,\na v_h)=(A\na \overline{v}_h,\na\overline{w})=(A\na \overline{v}_h,\na P_h\overline{w})=(A\na\overline{P_h\overline{w}},\na v_h)\quad \forall v_h\in V_h,
%\end{equation}
By imitating the analysis for elliptic projection in \cite[etc.]{zhuwu2013,wu2013,duwu2015,liwu2019} and using the interpolation estimates in \cite[etc.]{brenner2008,lenoir1986}, we have the following error estimates for all $w\in H_0^1(\D)\cap H^2(\OchO)$:
\begin{equation}\label{eq:errPh}
\LtD{w-P_h w}+ h\abs{w-P_h w}_1\ls h^2 \Ht{w}{\OchO},\quad \LtD{w-P_h^* w}+ h\abs{w-P_h^* w}_1\ls h^2 \Ht{w}{\OchO}. 
\end{equation}

\subsection{A discrete auxiliary problem of FEM}
Similarly to the analysis of the continuous nonlinear PML problem \eqref{eq:PML}, %\eqref{eq:PML:bound:all}, 
we introduce the FEM for the linear auxiliary problem \eqref{ap1}: %\eqref{ap2}
find $\hat w^\phi_h\in V_h$ such that
\begin{equation}\label{eq:dvap}
a^\phi (\hat w^\phi_h,v_h)%=(A\na \hat w^\phi_h,\na v_h)-k^2(B\hat w^\phi_h,v_h) - k^2\vep\big(\abs{\phi+\uinc}^2\hat w^\phi_h,v_h\big)_{\Omz} 
= (g,v_h)\quad \forall v_h\in V_h, 
\end{equation}
where $a^\phi$ is defined by
\begin{equation}\label{eq:aphi}
  a^\phi (u,v) := a(u,v) - k^2\vep \bgS{2\abs{\phi+\uinc}^2 u + (\phi+\uinc)^2 \overline{u}, v}_{\Om_0}.
\end{equation}
%:\[a^\phi (u,v) := a(u,v) - k^2\vep \bgS{2\abs{\phi+\uinc}^2 u + (\phi+\uinc)^2 \overline{u}, v}_{\Om_0}.\]
Note that, the variational formulation of the linear auxiliary problem \eqref{ap1} reads as: find $\hat w^\phi\in H_0^1(\D)$ such that
\begin{equation}\label{eq:vap}
a^\phi(\hat w^\phi,v) = (g,v)\quad \forall v\in H_0^1(\D).
\end{equation}

Next we give the error estimates between the auxiliary problems \eqref{eq:dvap} and \eqref{eq:vap}. %Let $\theta_0$ be a positive constant introduced in Lemma \ref{lem:stabap}.

\begin{lemma}\label{lem:errap}
Let the conditions of Lemma \ref{lem:stabl} be satisfied. There exist two positive constants $\theta_3\ls 1$ and $C_0$ such that if $k^3h^2\leq C_0$ and 
\eqn{\max\{k\vep\Linf{\phi}^2,\, k\vep\Linf{\uinc}^2\}\leq\theta_3,}
then the following error estimates hold:
\begin{equation} \label{eq:errap}
\big\Vert{\hskip -1pt}\big\vert \hat w^\phi-\hat w^\phi_h \big\vert{\hskip -1pt}\big\Vert \ls (kh+k^3h^2)\LtD{g}\qaq \big\Vert \hat w^\phi-\hat w^\phi_h\big\Vert_0 \ls k^2h^2\LtD{g}. 
\end{equation}
\end{lemma}
\begin{proof}
We can easily see the following Galerkin orthogonality for the error $e^\phi:= \hat w^\phi-\hat w^\phi_h$: 
\begin{equation}\label{eq:orap} 
a^\phi(e^\phi,v_h) = 0\quad \forall\, v_h\in V_h.
\end{equation}
Consider the dual system to the linear PML problem: find $z\in H^1_0(\D)$ such that
\begin{equation} 
%-\na\cdot(A^*\na z) - B^*k^2 z=e^\phi\quad \mbox{in }\D;\quad z=0\quad\mbox{on }\hGamma. 
a(v,z) = (v, e^\phi) \quad \forall\, v \in H^1_0(\D). \label{eq:dualap}
\end{equation}
It is obvious that $\hat L \overline{z} = \overline{e^\phi}$. From Lemma \ref{lem:stabl}, the following stability estimate for $z$ holds: 
\eqn{\He{z} + k^{-1} \Ht{z}{\OchO} \ls \LtD{e^\phi},} 
which together with \eqref{eq:dualap}, \eqref{eq:aphi}, \eqref{eq:orap} and \eqref{eq:Ph}--\eqref{eq:errPh} gives
\begin{align*}
\LtD{e^\phi}^2 =&\; a(e^\phi,z) = %(\na e^\phi,A^*\na z)-k^2\big(e^\phi,(B^*+\vep\oneo\abs{\phi+\uinc}^2)z\big) \\
 a^\phi(e^\phi,z) + k^2\vep \bgS{2\abs{\phi+\uinc}^2 e^\phi + (\phi+\uinc)^2 \overline{e^\phi}, z}_{\Om_0} \\
 =&\; a^\phi(e^\phi,z-P_h z) + k^2\vep \bgS{2\abs{\phi+\uinc}^2 e^\phi + (\phi+\uinc)^2 \overline{e^\phi}, z}_{\Om_0} \\
 =&\;\big( A\na (\hat w^\phi- I_h \hat w^\phi ),\na(z-P_h z) \big) - k^2\big(B e^\phi, z-P_h z\big) \\
 &\; + k^2\vep \big(2\abs{\phi+\uinc}^2 e^\phi + (\phi+\uinc)^2 \overline{e^\phi}, P_hz\big)_{\Omz} \\
 \ls&\; h^2 \Ht{\hat w^\phi}{\OchO}\Ht{z}{\OchO} + k^2 h^2 \LtD{e^\phi} \Ht{z}{\OchO} \\ 
 &\; + k^2\vep \bgS{\Linf{\phi}^2+\Linf{\uinc}^2} \LtD{e^\phi}\LtD{P_h z} \\
 \ls&\; k^2h^2\LtD{g}\LtD{e^\phi} + k^3h^2 \LtD{e^\phi}^2 + \ta_3 k \LtD{e^\phi}\LtD{P_h z},
\end{align*}
where $I_h$ is the standard finite element interpolation operator. Noting that 
\[k\LtD{P_h z} \leq k\LtD{z} + k\LtD{z- P_h z} \leq \He{z} + kh^2\Ht{z}{\OchO}\ls (1+k^2h^2)\LtD{e^\phi},\]
we arrive at
\[\LtD{e^\phi} \ls k^2h^2\LtD{g} + \Sp{k^3h^2+\ta_3(1+k^2h^2)} \LtD{e^\phi}.\]
If $k^3h^2\leq C_0$ and $\ta_3$ are both small enough, then the above result leads directly to the $L^2$-error estimate:
\begin{equation}\label{eq:L2errap} 
\LtD{e^\phi} \ls k^2h^2 \LtD{g}.
\end{equation}

Next we estimate $\LtD{\na e^\phi}$. Decomposing $e^\phi=(\hat w^\phi-P_h^*\hat w^\phi)-(\hat w^\phi_h-P_h^*\hat w^\phi)=:\eta-\xi_h$, we can then easily get from \eqref{eq:Ph} and \eqref{eq:orap} that 
\begin{align*}
\LtD{\na\xi_h}^2 &\eqsim\abs{(A\na\xi_h,\na\xi_h)} =\abs{(A\na e^\phi,\na\xi_h)} \\ 
&= k^2\big| \big( Be^\phi + 2\vep\oneo\abs{\phi+\uinc}^2 e^\phi + \vep\oneo(\phi+\uinc)^2 \overline{e^\phi},\xi_h\big) \big|  \\
&\ls (k^2+\theta_3 k)\LtD{e^\phi}\LtD{\xi_h}.
\end{align*}
Noting that $\LtD{\eta}\ls kh^2 \LtD{g}$ and $\LtD{\xi_h}\leq \LtD{\eta}+\LtD{e^\phi}\ls k^2h^2\LtD{g}$, we arrive at
\[ \LtD{\na\xi_h}^2 \ls(k^6h^4+\theta_3k^5h^4)\LtD{g}^2,
\]
and hence $\LtD{\na\xi_h}\ls k^3h^2\LtD{g}$, which together with \eqref{eq:errPh} implies that
\[ \LtD{\na e^\phi}\ls\LtD{\na\eta}+\LtD{\na\xi_h} \ls (kh+k^3h^2)\LtD{g}.
\]
The estimate with respect to the energy norm in \eqref{eq:errap} follows by noting that $\He{\cdot}\eqsim \LtD{\na \cdot}+k\LtD{\cdot}$.
\end{proof}

By combining Lemmas \ref{lem:errap} and \ref{lem:stabap}, we can obtain the stability estimate for $\hat w^\phi_h$.
\begin{corollary}\label{cor:stabdap}
  Under the conditions of Lemma \ref{lem:errap}, there holds
  \begin{equation}
  \big\Vert{\hskip -1pt}\big\vert \hat w^\phi_h\big\Vert{\hskip -1pt}\big\vert \ls \LtD{g}, \label{eq:stabdap}
  \end{equation}
and hence, the discrete auxiliary problelm \eqref{eq:dvap} is well-posed.
\end{corollary}

We end this subsection by giving an interior $L^\infty$-estimate for $\hat w^\phi_h$.

\begin{lemma}\label{lem:infdap}
Under the conditions of Lemma \ref{lem:errap}, there holds
\begin{equation}
\big\Vert \hat w^\phi_h \big\Vert_{L^\infty(\Omz)} \ls k^{\frac{d-3}2}\abs{\ln h}^{\bar{d}}\LtD{g} 
\quad {\rm with} ~~\bar{d}= \begin{cases}
  0, & d=2,\\
  1, & d=3.
\end{cases}
\label{eq:infdap}
\end{equation}
\end{lemma}
\begin{proof}
For the estimate \eqref{eq:infdap}, we write $\eta_h=\hat w^\phi_h-P_h^*\hat w^\phi$, then we have 
by the triangle inequality that 
\eq{\label{ewh} \big\Vert \hat w^\phi_h \big\Vert_{L^\infty(\Omz)} \leq \norm{\eta_h}_{L^\infty(\Omz)} + \big\Vert P_h^*\hat w^\phi - \hat w^\phi\big\Vert_{L^\infty(\Omz)} + \big\Vert \hat w^\phi \big\Vert_{L^\infty(\Omz)}, 
}
and it suffices to estimate three terms on the right-hand side above. 

First, $\Linf{\hat w^\phi}$ was already given in Lemma~\ref{lem:linfap}.
%, we have 
%\[ \Linf{\hat w^\phi}\ls k^{\frac{d-3}{2}}\LtD{g}.
%\] 

Next, we estimate $\big\Vert P_h^*\hat w^\phi - \hat w^\phi\big\Vert_{L^\infty(\Omz)}$ by considering two 
and three dimensional cases separately.  For the two dimensional case, by using the $L^\infty$-estimate for the FE interpolation (see e.g. \cite[(4.4.8)]{brenner2008}), the inverse estimate (see e.g. \cite[(4.5.4)]{brenner2008}),  the a priori estimate in \eqref{eq:stabap}, the $L^2$-error estimate for the elliptic projection in \eqref{eq:errPh}, and noting that $k^3h^2\ls 1$ , we get
\begin{equation}\label{eq:linf1}
\begin{aligned}
 % \big\Vert \hat w^\phi-P_h^*\hat w^\phi \big\Vert_{L^\infty(\Omz)} &\ls  \big\Vert \hat w^\phi-I_h\hat w^\phi \big\Vert_{L^\infty(\Omega_0)}+\big\Vert I_h\hat w^\phi-P_h^*\hat w^\phi \big\Vert_{L^\infty(\Omz)} \\
  %&\ls h^{2-\frac{d}p} \big\vert\hat w^\phi\big\vert_{W^{2,p}(\Omega_1)} + h^{-\frac d2}\big\Vert\hat w^\phi-I_h\hat w^\phi-\big(\hat w^\phi-P_h^*\hat w^\phi\big)\big\Vert_{0,\Omega_1} \\
  %&\ls h^{2-\frac{d}p}\big(k^2\big\Vert \hat w^\phi\big\Vert_{L^p(\Omega_2)} + \Vert g\Vert_{L^p(\Omega_2)}\big) + h^{2-\frac d2}\big\vert \hat w^\phi\big\vert_{2,\Omega_1} + kh^{2-\frac d2}\Vert g\Vert_0 \\
  %&\ls h^{2-\frac dp}\big(k^2\big\Vert \hat w^\phi\big\Vert^{\frac 2p}_{0,\Omega_2}\big\Vert \hat w^\phi\big\Vert^{1-\frac 2p}_{L^\infty(\Omega_2)} + \Vert g\Vert_{L^p (\Omega_2)}\big) + kh^{2-\frac d2}\Vert g\Vert_0 \\
  %&\ls k^{\frac{d-3}{2}}\big((k^3h^2)^{\frac{2p-d+1}{3p}}h^{\frac{2p-d-2}{3p}}\Vert g\Vert_0+(k^3h^2)^{\frac{-d+5}{6}}h^{\frac{2-d}{6}}\Vert g\Vert_0 \\
  %&\quad + (k^3h^2)^{\frac{-d+3}6}h^{1+(\frac 13-\frac 1p)d}\Vert g\Vert_{L^p(\Omega)}\big) \\
  %&\ls k^{\frac{d-3}2}\big((h^{\frac{2p-d-2}{3p}}+h^{\frac{2-d}{6}})\LtD{g} + h^{1+(\frac 13 -\frac 1p)d}\Vert g\Vert_{L^p(\Om)}\big),
  \big\Vert \hat w^\phi-P_h^*\hat w^\phi \big\Vert_{L^\infty(\Omz)} &\ls  \big\Vert \hat w^\phi-I_h\hat w^\phi \big\Vert_{L^\infty(\Omega_0)}+\big\Vert I_h\hat w^\phi-P_h^*\hat w^\phi \big\Vert_{L^\infty(\Omz)} \\
  &\ls h\big\Vert\hat w^\phi\big\Vert_{2,\Omega_0} + h^{-1}\big\Vert\hat w^\phi-I_h\hat w^\phi-\big(\hat w^\phi-P_h^*\hat w^\phi\big)\big\Vert_{0,\Omega_0} \\
  &\ls  h\big\Vert \hat w^\phi\big\Vert_{2,\Omega_0} + h \Ht{\hat w^\phi}{\OchO}\ls kh\Vert g\Vert_0 = k^{-\frac{1}{2}}(k^3h^2)^{\frac{1}{2}}\Vert g\Vert_0 \\
  & \ls k^{-\frac{1}2}\Vert g\Vert_0.
\end{aligned}
\end{equation}
%where $\Omz \subset\subset \Omega_1 \subset\subset \Omega_2 \subset\subset \Omega$ and $\Omega_1$ contains all the elements which intersect with $\Omega_0$ and $\dist(\Ga,\Omega_2) \eqsim \dist(\partial\Omega_2,\partial\Omega_1) $. 
For the case of $d=3$, we consider a subdomain $\Om_1$
satisfying that $\Omz \subset\subset \Omega_1  \subset\subset \Omega$ and $\dist(\Ga,\Omega_1) \eqsim \dist(\partial\Omega_1,\partial\Omega_0) $, then 
we can use the interior $L^\infty$-error estimates (see \cite[Theorem 5.1]{schatz1977}), \cb{\eqref{eq:errPh}} and  Remark~\ref{rm:stabap} to get
\begin{equation} \label{eq:linfd3}
  \begin{aligned}
    \big\Vert \hat w^\phi-P_h^*\hat w^\phi \big\Vert_{L^{\infty}(\Omz)}&\ls \abs{\ln h}\big\Vert \hat w^\phi\big\Vert_{L^{\infty}(\Omega_1)} + \big\Vert \hat w^\phi-P_h^*\hat w^\phi\big\Vert_0 \\
    &\ls \abs{\ln h}\big\Vert \hat w^\phi\big\Vert_{L^{\infty}(\Omega_1)} + h\big\vert \hat w^\phi \big\vert_1\ls \big(\abs{\ln h}+h\big)\Vert g\Vert_0 \\
     &\ls \abs{\ln h}\Vert g\Vert_0, 
  \end{aligned}
\end{equation}
   
It remains to estimate $\norm{\eta_h}_{L^\infty(\Omz)}$ in \eqref{ewh}. 
From \eqref{eq:Ph} and \eqref{eq:orap}, %\cb{\sout{and \eqref{eq:dvap}}}, 
we have for any $v_h\in V_h$, 
\begin{align*}
(A\na \eta_h,\na v_h) &= \big(A\na(\hat w^\phi_h-\hat w^\phi),\na v_h\big) \\
&= k^2 \big( (B+2\vep\oneo\abs{\phi+\uinc}^2)(\hat w^\phi_h-\hat w^\phi) + \vep\oneo(\phi+\uinc)^2\overline{(\hat w^\phi_h-\hat w^\phi)}, v_h\big).
\end{align*}
We can easily see that $\eta_h$ is the finite element approximation of 
the solution $\eta\in H_0^1(\D)$ to the system
\begin{equation} \label{eq:eta}
  \begin{aligned}
  -\na\cdot (A\na \eta) &= k^2 \big( (B+2\vep\oneo\abs{\phi+\uinc}^2)(\hat w^\phi_h-\hat w^\phi) + \vep\oneo(\phi+\uinc)^2\overline{(\hat w^\phi_h-\hat w^\phi)} \big) \quad\text{in }\D.%; \quad \eta = 0 \quad \text{on }\hat\Gamma.
  \end{aligned}
\end{equation}
%\begin{align*}
%(A\na \eta,\na v_h) = (g,v_h)=(A\na \eta_h,\na v_h) \quad\forall v_h\in V_h,
%\end{align*}
%$\eta_h$ is the finite element solution of $\eta$. 
%\cb{\sout{with the boundary condition $\eta=0$ on $\hGamma$.}} 
By using the a priori estimate and Lemma \ref{lem:errap}, we get
\begin{align}\label{Eetah}
\begin{split}
&\Ht{\eta}{\OchO} \ls  k^2 \bgS{1+\vep\Linf{\phi}^2+\vep\Linf{\uinc}^2} \big\Vert (\hat w^\phi_h-\hat w^\phi) \big\Vert_0 \ls k^4h^2\LtD{g}, \\
&\LtD{\eta-\eta_h} \ls h^2\Ht{\eta}{\OchO} \ls  k^4h^4\LtD{g}.
\end{split}
\end{align}
On the other hand, we know from \eqref{eq:eta} that $\eta \in H_0^1(\D)$ solves the linear PML equation
\begin{align*}
\hat L \eta = Bk^2(\eta_h-\eta) + Bk^2 (P_h^*\hat w^\phi-\hat w^\phi) + k^2\vep\oneo\bgS{2\abs{\phi+\uinc}^2(\hat w^\phi_h-\hat w^\phi) + (\phi+\uinc)^2\overline{(\hat w^\phi_h-\hat w^\phi)}}.
%(A\na \eta, \na v) - k^2(B\eta, v) =&\; k^2\big(B(\eta_h-\eta),v\big) +k^2\big(B(P_h^*\hat w^\phi-\hat w^\phi), v\big) \\
%&\;+ \cb{k^2\vep \big(2\abs{\phi+\uinc}^2(\hat w^\phi_h-\hat w^\phi) + (\phi+\uinc)^2\overline{(\hat w^\phi_h-\hat w^\phi)}, v\big)_{\Om_0}} 
\end{align*}
%\cb{for any $v\in H^1_0(\D)$.} 
By combining \eqref{Linf}, \eqref{Eetah}, \eqref{eq:errPh} and Lemma \ref{lem:errap}, and noting that $k^3h^2\ls 1$, we can derive 
\begin{equation}\label{eq:Linfeta}
\begin{aligned} 
\Vert \eta \Vert_{L^\infty(\Omz)} &\ls k^{\frac{d-3}2} k^2\big( \Vert\eta_h-\eta \Vert_0+ \big\Vert P_h^*\hat w^\phi-\hat w^\phi\big\Vert_0+k^{-1}\big\Vert \hat w^\phi_h-\hat w^\phi \big\Vert_0 \big) \\
&\ls k^{\frac{d-3}2}(k^6h^4+k^3h^2+k^3h^2) \LtD{g} 
\ls  k^{\frac{d-3}2}\LtD{g}.
\end{aligned}
\end{equation}
Furthermore, using the interior $L^\infty$-error estimates, \eqref{Eetah}, the interpolation error estimates, the interior $W^{2,p}$-estimates (see \cite[Theorem 9.11]{gilbarg2001}) for the elliptic problem \eqref{eq:eta}, and the Sobolev embedding  $H^1 \hookrightarrow L^q$ (with $q=6$ for $d=3$ and $q=7$ for $d=2$), we can deduce 
\begin{align*} 
\Vert \eta-\eta_h \Vert_{L^\infty(\Omz)} &\ls \abs{\ln h} \Vert \eta-I_h\eta \Vert_{L^\infty(\Omega_1)}+\LtD{\eta-\eta_h} \\
 &\ls \abs{\ln h} h^{2-\frac{d}{q}}\norm{\eta}_{W^{2,q}(\Omega_1)} +k^4h^4\LtD{g} \\
 &\ls \abs{\ln h} h^{2-\frac{d}{q}}\big(\Vert \eta\Vert_{L^{q}(\Omega)} + k^2\big\Vert \hat w^\phi_h-\hat w^\phi \big\Vert_{L^{q}(\Om)}\big) + k^4h^4\LtD{g} \\
 &\ls \abs{\ln h} h^{2-\frac{d}{q}}\big(\Vert \eta\Vert_{2,\OchO} + k^2\big\Vert \hat w^\phi_h-\hat w^\phi \big\Vert_{1}\big)  + k^4h^4\LtD{g}\\
 &\ls \big( \abs{\ln h}  h^{2-\frac{d}{q}}(k^4h^2+k^2(kh+k^3h^2)) +k^4h^4 \big) \LtD{g}\\
 %&\ls ( \abs{\ln h}  h^{2-\frac{d}{q}} k^2+ k^{-2}) \LtD{g} \\
 &\ls \big( \abs{\ln h} h^{\frac{d-1}3-\frac dq} (k^3h^2)^{\frac{7-d}{6}} + k^{-\frac{d+1}2}\big)k^{\frac{d-3}2}\LtD{g} \\
 &\ls k^{\frac{d-3}2}\LtD{g},
\end{align*}
where we have used ${(d-1)}/3- d/q>0$ in the last inequality. This together with \eqref{eq:Linfeta} yields 
\[ \Linf{\eta_h} \leq \Linf{\eta} + \Linf{\eta-\eta_h} \ls k^{\frac{d-3}2}\LtD{g}.
\]
Now the desired estimate \eqref{eq:infdap} is a consequence of the above estimate, 
\eqref{eq:linf1}--\eqref{eq:linfd3}, and \eqref{eq:Linfap}. 
\end{proof}

\begin{remark}
  In \cite[Lemma 3.4]{wuzou2018}, a similar interior  estimate to \eqref{eq:infdap} 
  was established for the FE solution to a linear auxiliary problem with impedance boundary condition, that is, the  $L^\infty(\Om_0)$ norm of the FE solution is bounded by $O(\abs{\ln h} k^{\frac{d-3}2})$ for both the two and three dimensional cases. But we note that in two dimensions, the new estimate \eqref{eq:infdap} for the linear auxiliary problem with PML boundary condition improves the the previous estimate by removing the logarithmic factor in $h$. However, if we use the same technique as we did 
in \eqref{eq:linf1} to deal with the case of $d=3$, we will get the estimate
  \[
    \| \hat w^\phi - P_h^* \hat w^\phi \|_{L^\infty(\Omega_0)} \lesssim h^{\frac12} \| \hat w^\phi \|_{2,\Omega\cup\hat\Omega} \lesssim kh^{\frac12}\|g\|_0 \lesssim {h^{-\frac16}} \|g\|_0,\quad \mbox{if }k^3h^2 \lesssim 1,
  \]
which is obviously much worse than the estimate \eqref{eq:linfd3}. This is the main reason why we have separated 
the case of $d=3$ from the case of $d=2$ in our analysis. 

  %since $\abs{\ln h} \ll h^{-1/6}$ when $h$ is small.
  %  Wu and Zou proved a similar estimate for the nonlinear Helmholtz problem with impedance boundary condition $\partial_n u+\i k u=g$ on $\partial\Om$:
  %  \[ \Linf{u^{\rm aux}_h} \ls \abs{\ln h} k^{\frac{d-3}2}(\Vert f\Vert_{0,\Om}+\Vert g\Vert_{0,\partial\Om}+k^{-1}\Vert g\Vert_{\frac12,\partial\Om}),
  %  \]
  %where $u^{\rm aux}_h$ is the finite element solution to some corresponding auxiliary problem. We remark that, our result improves a little compared with the previous one by removing the logarithmic factors in $h$ before $k^{\frac{d-3}2}$ for d=2. However, for the case of $d=3$, if we use the same method for d=2, we get the result which the order of $h$ is $-\frac 16$. It is worse than the term $\ln h$.
\end{remark}

\subsection{Preasymptotic error estimates}\label{sec:pee}
Like the analyses in Section \ref{sec:anas} for the continuous nonlinear PML system 
\eqref{eq:PML}, %\eqref{eq:PML:bound:all}
we consider an Newton's iterative procedure to approach the solution $u_h$ to the nonlinear FEM \eqref{eq:FEM}
(with $a^\phi(\cdot,\cdot)$ defined as \eqref{eq:aphi}): 

For a given $u_h^0\in V_h$, find $u_h^{l+1}\in V_h$ for $l=0,1,2,\cdots$, such that
\begin{equation}\label{eq:dip}
a^{u_h^l}(u_h^{l+1},v_h)= (f,v_h)_{\Om} - k^2\vep\big(2\,|u_h^l+\uinc|^2 u_h^l - (u_h^l+\uinc)^2 \overline{\uinc},v_h\big)_{\Omz} \quad\forall v_h\in V_h\,.
\end{equation}

%The following results can be proved by applying Corollary \ref{cor:stabdap} and Lemma \ref{lem:infdap} and imitating the procedure of Lemma \ref{lem:pmlip}.
We first give the stability estimates for the discrete solutions $u_h^l$ with $l\geq 1$, 
whose proof follows from the one of Lemma~\ref{lem:pmlip}, except for using Corollary~\ref{cor:stabdap} and Lemma~\ref{lem:infdap} instead of Lemmas~\ref{lem:stabap}-\ref{lem:linfap}, respectively. 
%by imitating the procedure of Lemma \ref{lem:pmlip}. 
\begin{lemma}
Let the conditions of Lemma \ref{lem:stabl} be satisfied, and there exists a positive constant $\ta_4\ls 1$  such that 
\begin{align}
&k\norm{u_h^0}_{0,\Omega_0}\ls\Mf,\quad \Linf{u_h^0}\ls k^{\frac{d-3}2}\abs{\ln h}^{\bar{d}}\Mf,\quad \label{eq:cond4-2} \\
&\max\Big\{ k^{d-2}\vep\abs{\ln h}^{2\bar{d}}\Mf^2,~k\vep\Linf{\uinc}^2\Big\} \leq \ta_4,\label{eq:cond4-3}
\end{align}
then the following estimates hold, under the condition that 
$k^3h^2\leq C_0$:
\begin{equation} \label{eq:lem4-5-1}
\He{u_h^l} \ls \Mf \qaq \Linf{u_h^l} \ls k^{\frac{d-3}2}\abs{\ln h}^{\bar{d}} \Mf \quad 
\mbox{for} ~~l=1,2,\cdots 
\end{equation}
\end{lemma}

By taking the limit $l\to \infty$ and following the proof of Theorem \ref{thm:stabpml}, we can obtain the following stability estimates of the FE solution $u_h$ to \eqref{eq:FEM}. %\cb{\sout{and the convergence rate of the iterative scheme \eqref{eq:dip}.}}
\begin{theorem}\label{thm:fem_estimates}
Let the conditions of Lemma \ref{lem:stabl} be satisfied, there exists a positive constant $\ta_5 \ls 1$ %$\widetilde{C}_2$  
such that if $k^3h^2\leq C_0$ (with $C_0$ is from Lemma~\ref{lem:errap}) and 
\begin{equation}\label{eq:FEM_conditions}
%\rho_h:=\widetilde{C}_2\Big(k^{d-2}\vep\abs{\ln h}^{2\bar{d}}\Mf^2 +k\vep\Linf{\uinc}^2\Big) <1,
\max\Big\{ k^{d-2}\vep\abs{\ln h}^{2\bar{d}}\Mf^2,~k\vep\Linf{\uinc}^2\Big\} \leq \ta_5,
\end{equation} 
then the FEM \eqref{eq:FEM} attains a unique solution $u_h$ satisfying the estimates:
\begin{equation} \label{eq:stab:uh}
\He{u_h} \ls \Mf \qaq \Linf{u_h} \ls k^{\frac{d-3}2}\abs{\ln h}^{\bar{d}}\Mf.
\end{equation}
%\sout{Moreover, the following error estimate holds under the conditions \eqref{eq:cond4-2}:}
%\begin{equation}\label{eq:conv}
%\sout{\He{u_h-u_h^l}\ls \rho_h^l\Mf.}
%\end{equation}
\end{theorem}

We can naturally get an error estimate between the NLH problem \eqref{eq:Helm}--\eqref{eq:Somm} and the FEM \eqref{eq:FEM}, which follows by applying Theorem~\ref{thm:conver}, Lemma \ref{lem:errap} and taking $l\to\infty$ in \eqref{ip1} %\eqref{ip2} 
and \eqref{eq:dip}, respectively. 
%The proof is quite similar to that of Theorem \ref{thm:conver} or \cite[Theorem 3.7]{wuzou2018} 
%and we omit the details.
\begin{theorem}\label{thm:err}
  Under the conditions of Lemma \ref{lem:stabl}, there exist constants $C_1, C_2, C_3, \ta>0$ such that if   $k^3h^2\le C_0$ (with $C_0$ is from Lemma~\ref{lem:errap}) and 
  \eq{\label{cond:err} \max\Big\{ k^{d-2}\vep\abs{\ln h}^{2\bar{d}}\Mf^2,~k\vep\Linf{\uinc}^2\Big\} \leq\theta, }
  then the FE solution $u_h$ to \eqref{eq:FEM} approximates the NLH solution $u$ to \eqref{eq:Helm}--\eqref{eq:Somm}, with the error estimate
  \begin{equation}\label{eq:err}
  \He{u-u_h}_\Om\le (C_1kh+C_2k^3h^2)\Mf+C_3k^5 e^{-2 k\siz L\big(1-\frac{R^2}{\hR^2+\siz^2L^2}\big)^{1/2}}\Mf.
  \end{equation} 
  \end{theorem}
  
  Like Remark \ref{rmk:low-energy}, we remark that Theorem \ref{thm:err} says that the FE solution given by \eqref{eq:FEM} approximates the low-energy solution to NLH if \eqref{cond:err} holds so that the nonlinearity is not too strong. While, this result has not excluded the possibility of multiple solutions to FEM \eqref{eq:FEM}. 
  More specifically,  the following theorem shows that the FE solution $u_h^l$ given by Newton’s iteration \eqref{eq:dip} converges quadratically if the initial guess is close to one of the FE solutions to \eqref{eq:FEM} even for problems with strong nonlinearity. 

\begin{theorem}%[\cb{Version II}]
Let $\hu^*$ be one of the multiple solutions to the nonlinear PML problem \eqref{eq:PML} and $u_h^*$ be the corresponding FE solution given by \eqref{eq:FEM}. Let 
$\tilde w^* \in H_0^1(\D)$ be the solution to the dual problem of \eqref{eq:ws}:
\eq{\label{eq:dws} \hat L \tilde w^* - k^2\vep \oneo \bgS{2\abs{\hu^* + \uinc}^2 \tilde w^* + \overline{(\hu^* + \uinc)}^2 \overline{\tilde w^*}} = \tilde g^*.}
Suppose the following stability estimate holds for any given function $\tilde g^*\in L^2(\D)$,
\eq{\label{eq:stab:dws}
\He{\tilde w^*} \leq \tilde C^* \norm{\tilde g^*}_0,}
where $\tilde C^*\gtrsim1$ may depend on $k$ and $\hu^*$.
Denote by
\eqn{\gamma^* &:=  \min\big\{ \big(48 \tilde C_1^*k\vep \Linf{\hu^*+\uinc}\big)^{-1}, \big(24 \tilde C_1^*k\vep\big)^{-\frac12} \big\},\\
\tilde N^* &:= \tilde C_1^* C_{\rm Nir}^2k^{\frac{d}{2}} \vep  \bgS{20 \Linf{\hu^*+\uinc} + 12  \gamma^*}, \qaq \\
\tilde \ta^* &:= \min \BgA{\big(\tilde N^*\big)^{-1}, C_{\rm Nir}^{-1} k^{1-\frac{d}{2}}\Linf{\hu^*+\uinc}},}
where $\tilde C_1^*$ is a constant satisfying $\tilde C_1^* \eqsim \tilde C^*\big(1+\vep\Linf{\hu^*+\uinc}^2\big)^{\frac12}$.
Suppose the FE solution $u_h^*$ satisfies $\Linf{\hu^*-u_h^*} \leq \gamma^*$ and the initial guess $u_h^0$ satisfies $\He{u_h^0-u_h^*}\le \tilde\ta^*$,
there exists a positive constant $\tilde C_0$ such that if
\eq{\label{eq:h0} k^3h^2\big(1+\vep \Linf{\hu^*+\uinc}^2\big)^2\tilde C^*\le \tilde C_0,}
 then the Newton's iterative sequence $\{u_h^l\}_{l\geq 1}$ defined by \eqref{eq:dip} converges quadratically to $u_h^*$, namely,
\begin{equation}\label{eq:dquad}
  \He{u_h^{l+1} - u_h^*} \ls \tilde N^* \He{u_h^l - u_h^*}^2, \quad l =0, 1, 2, \cdots.
\end{equation}
\end{theorem}
\begin{proof}
Denote $e_h^l = u_h^l -u_h^*$ and $e_h^* = u_h^* - \hu^*$. Similar to \eqref{NT1}, from \eqref{eq:FEM} and \eqref{eq:dip}, it can be verified that
  \eqn{a^{u_h^l}(e^{l+1}_h,v_h) = -k^2\vep\bgS{2\,|{e_h^l}|^2 (u_h^*+\uinc+e_h^l) + (e_h^l)^2\overline{(u_h^*+\uinc)},v_h}_{\Omz} \quad\forall\, v_h \in V_h,}
which can be rewritten as   
\eqn{ a^{\hu^*}(e^{l+1}_h,v_h) =& -k^2\vep(G_1 - G_2, v_h)_{\Omz}\quad\forall v_h\in V_h,\quad\text{where}\\
G_1 =&\; 2\,|{e_h^l}|^2 (\hu^*+\uinc) + 2\,|e_h^l|^2 e_h^l + 2\,|e_h^l|^2 e_h^* +(e_h^l)^2\overline{(\hu^*+\uinc)} + (e_h^l)^2\overline{e_h^*}, \\
G_2 =&\; \BgS{2\,|e_h^l|^2 + 2\,|e_h^*|^2 + 4 \Re\bgS{{e_h^l}\overline{(\hu^*+\uinc)} + e_h^*\overline{(\hu^*+\uinc)} + e_h^l \overline{e_h^*}}} e_h^{l+1} \\
&\;+ \BgS{(e_h^l)^2 + (e_h^*)^2 + 2 \bgS{{e_h^l}{(\hu^*+\uinc)} + {e_h^*}(\hu^*+\uinc) + e_h^l {e_h^*}}} \overline{e_h^{l+1}}.
}  
Clearly, $e_h^{l+1}\in V_h$ is the FE solution to the linearized PML problem \eqref{eq:ws} with $g^* = -k^2\vep\oneo (G_1-G_2)$. 
By applying \eqref{eq:stab:dws} and Lemma \ref{app:dstab} with 
\eqn{p &=-Bk^2 - 2k^2 \vep\oneo \abs{\hu^*+\uinc}^2, \;\; q=-k^2\vep\oneo (\hu^*+\uinc)^2, \;\; \tilde c_0 \eqsim k^{-1}\tilde C^*,
}
and noting that $\norm{p}_{L^\infty(\D)} + \norm{q}_{L^\infty(\D)}\ls k^2+k^2 \vep\Linf{\hu^*+\uinc}^2$ and the fact that \eqref{eq:h0} implies \eqref{app:h0}, it follows from \eqref{app:stab:wh} that
\eq{ \label{eq:dstab}\He{e_h^{l+1}} \leq\tilde C_1^* \norm{g^*}_0, \quad\text{where}\quad \tilde C_1^* \eqsim \tilde C^*\big(1+\vep\Linf{\hu^*+\uinc}^2\big)^{\frac12}.
}
Applying \eqref{ni}, we have
\eqn{\norm{G_1}_{0,\Omz} &\leq 3 \Linf{\hu^* +\uinc} \norm{e_h^l}_{L^4(\Omz)}^2 + 2\norm{e_h^l}_{L^6(\Omz)}^3 + 3 \Linf{e_h^*} \norm{e_h^l}_{L^4(\Omz)}^2 \\
&\leq C_{\rm Nir}^2  \BgS{3k^{\frac{d}2-2}\Linf{\hu^*+\uinc} + 2 C_{\rm Nir} k^{d - 3}\He{e_h^l} + 3 k^{\frac{d}2-2}\gamma^*}  \He{e_h^l}^2\,.
}
%\eqn{ \norm{G_1}_{0,\Omz} &\leq 3 \norm{\hu^* +\uinc}_{L^6(\Omz)} \norm{e_h^l}_{L^6(\Omz)}^2 + 2\norm{e_h^l}_{L^6(\Omz)}^3 + 3 \Linf{e_h^*} \norm{e_h^l}_{L^4(\Omz)}^2 \\
%&\leq C_{\rm Nir}^2 k^{\frac23 d-2} \BgS{3\norm{\hu^*+\uinc}_{L^6(\Omz)} + 2 C_{\rm Nir} k^{\frac d3 - 1}\He{e_h^l} + 3 k^{-\frac d6} \gamma^*} \He{e_h^l}^2,
%}
Similarly, we can derive 
\eqn{ \norm{G_2}_{0,\Omz} \leq&\; 3\norm{e_h^l}_{L^6(\Omz)}^2 \norm{e_h^{l+1}}_{L^6(\Omz)} + 3\Linf{e_h^*}^2 \norm{e_h^{l+1}}_{0,\Omz} + 6 \Linf{e_h^*} \norm{e_h^l}_{L^4(\Omz)} \norm{e_h^{l+1}}_{L^4(\Omz)}\\
&\; + 6 \Linf{\hu^*+\uinc} \norm{e_h^l}_{L^4(\Omz)} \norm{e_h^{l+1}}_{L^4(\Omz)} + 6 \Linf{e_h^*} \Linf{\hu^*+\uinc} \norm{e_h^{l+1}}_{0,\Omz} \\
\leq&\;  C_{\rm Nir}^2 \BgS{3 C_{\rm Nir} k^{d - 3}\He{e_h^l} + 6 k^{\frac{d}2-2} \gamma^*+6 k^{\frac{d}2-2} \Linf{\hu^*+\uinc} } \He{e_h^l} \He{e_h^{l+1}} \\
&\; +    3 k^{-1}\gamma^* \big(\gamma^* + 2 \Linf{\hu^*+\uinc}\big) \He{e_h^{l+1}}.
}
%\eqn{ \norm{G_2}_{0,\Omz} \leq&\; 3\norm{e_h^l}_{L^6(\Omz)}^2 \norm{e_h^{l+1}}_{L^6(\Omz)} + 3\Linf{e_h^*}^2 \norm{e_h^{l+1}}_{0,\Omz} + 6 \Linf{e_h^*} \norm{e_h^l}_{L^4(\Omz)} \norm{e_h^{l+1}}_{L^4(\Omz)}\\
%&\; + 6 \norm{\hu^*+\uinc}_{L^6(\Omz)} \norm{e_h^l}_{L^6(\Omz)} \norm{e_h^{l+1}}_{L^6(\Omz)} + 6 \Linf{e_h^*} \Linf{\hu^*+\uinc} \norm{e_h^{l+1}}_{0,\Omz} \\
%\leq&\;  C_{\rm Nir}^2 k^{\frac23 d-2} \BgS{6 \norm{\hu^*+\uinc}_{L^6(\Omz)} + 3 C_{\rm Nir} k^{\frac d3 - 1}\He{e_h^l} + 6 k^{-\frac d6} \gamma^*} \He{e_h^l} \He{e_h^{l+1}} \\
%&\; + 3 k^{-1}\gamma^* \big(\gamma^* + 2 \Linf{\hu^*+\uinc}\big) \He{e_h^{l+1}}.
%}
Applying \eqref{eq:dstab} and supposing $\He{e_h^l} \leq \tilde\ta^* $, we get
\eqn{ \He{e_h^{l+1}} \leq&\; \tilde C_1^* k^2 \vep \bgS{\norm{G_1}_{0,\Omz} + \norm{G_2}_{0,\Omz}} \\
\leq&\; \tilde C_1^* C_{\rm Nir}^2 \vep\bgS{5k^{\frac{d}2}\Linf{\hu^*+\uinc} + 3 k^{\frac{d}2}\gamma^*} \bgS{\He{e_h^l}^2 + 2 \He{e_h^l}\He{e_h^{l+1}}} \\
&\; + \big(3 C_1^*  k\vep(\gamma^*)^2 + 6  C_1^* k\vep \Linf{\hu^*+\uinc}\gamma^*\big) \He{e_h^{l+1}} \\
\leq &\; \tfrac14 \tilde N^* \He{e_h^l}^2 + \big(\tfrac12 \tilde N^*\tilde\ta^* + \tfrac14\big) \He{e_h^{l+1}},
}
which yields $\He{e_h^{l+1}} \leq \tilde N^* \He{e_h^l}^2$ and moreover $\He{e_h^{l+1}} \leq \tilde N^* (\tilde\ta^*)^2 \le \tilde\ta^*$. Then the proof of \eqref{eq:dquad} follows by induction.
\end{proof}

\begin{remark}
Similarly to Remark \ref{rmk:quad-low}, under the conditions of Theorem~\ref{thm:stabpml}, the low-energy solution $\hu^*=\hu$ satisfies the stability condition \eqref{eq:stab:dws} with $\tilde C^* \eqsim 1$ and hence, $\tilde C_1^*\eqsim 1$. The mesh condition \eqref{eq:h0} is satisfied if $k^3h^2$ is sufficiently small.
Moreover, let $\Theta=\max\big\{ k^{d-2}\vep\abs{\ln h}^{2\bar{d}}\Mf^2,~k\vep\Linf{\uinc}^2\big\}$, by following the proofs of \eqref{eq:stab} and \eqref{eq:stab:uh}, we can get 
\eqn{ &k\vep \Linf{\hu^*-u_h^*}^2 \leq 2k\vep \big(\Linf{\hu^*}^2 + \Linf{u_h^*}^2\big) \leq 4\tilde C_1 k^{d-2}\vep\abs{\ln h}^{2\bar{d}}\Mf^2 \leq 4\tilde C_1\Theta \quad \mbox{and}}
%for some constant $\tilde C_1$ and 
\eqn{
&k\vep \Linf{\hu^*-u_h^*} \Linf{\hu^*+\uinc} \leq k\vep \big(2\Linf{\hu^*}^2 + \Linf{u_h^*}^2 + \Linf{\uinc}^2\big) \leq (3\tilde C_1+1)\Theta %\\
%&\qquad\qquad\leq\tilde C_2\big( k^{d-2}\vep\abs{\ln h}^{2\bar{d}}\Mf^2 + k\vep\Linf{\uinc}^2 \big) \leq (48\tilde C_1^*)^{-1},
}
for some constant $\tilde C_1$. If $\Theta$ is sufficiently small such that $4\tilde C_1\Theta \leq (24\tilde C_1^*)^{-1}$ and $(3\tilde C_1 + 1)\Theta \leq (48\tilde C_1^*)^{-1}$, then the above two estimates yield
\eqn{ \Linf{\hu^*-u_h^*} \leq (24 \tilde C_1^* k \vep)^{-1/2} \qaq \Linf{\hu^*-u_h^*} \leq \big(48\tilde C_1^* k\vep \Linf{\hu^*+\uinc}\big)^{-1}.
}
Then, noting the definition of $\gamma^*$, the condition $\Linf{\hu^*-u_h^*} \leq \gamma^*$ is also satisfied.
%if the maximum between 
%$k^{d-2}\vep\abs{\ln h}^{2\bar{d}}\Mf^2$ and $k\vep\Linf{\uinc}^2$  
%%$\max\big\{ k^{d-2}\vep\abs{\ln h}^{2\bar{d}}\Mf^2,~k\vep\Linf{\uinc}^2\big\}$ 
%is sufficiently small, we can get
%\eqn{ &k\vep \Linf{\hu^*-u_h^*}^2 \leq k\vep \big(\Linf{\hu^*}^2 + \Linf{u_h^*}^2\big) \leq 2\tilde C_1 k^{d-2}\vep\abs{\ln h}^{2\bar{d}}\Mf^2 \leq (24\tilde C_1^*)^{-1},\quad\mbox{and} \\
%&k\vep \Linf{\hu^*-u_h^*} \Linf{\hu^*+\uinc} \leq k\vep \big(\Linf{\hu^*-u_h^*}^2 + \Linf{\hu^*+\uinc}^2\big)/2 \\
%&\qquad\qquad\leq\tilde C_2\big( k^{d-2}\vep\abs{\ln h}^{2\bar{d}}\Mf^2 + k\vep\Linf{\uinc}^2 \big) \leq (48\tilde C_1^*)^{-1},
%}
%which yields the condition\footnote{I don't understand the above estimate and also how to get this conclusion. 
%It would be helpful to make it clearer.} 
%$\Linf{\hu^*-u_h^*} \leq \gamma^*$ is also satisfied. 
Therefore, $\{u_h^l\}_{l\geq 1}$, defined by \eqref{eq:dip}, converges quadratically to $u_h^*$, as long as $u_h^0$ is close enough to $u_h^*$. 
\end{remark}

\subsection{CIP-FEM}
It is well known that the linear Helmholtz equation with high wave number suffers from the pollution effect. 
Extensive studies have been carried out for estimating and reducing the pollution effect in the literature (see, e.g., \cite{babuska2000,fengwu2009,hiptmair2009,melenk2010}). The CIP-FEM, which was first proposed 
%by Douglas and Dupont 
in \cite{douglas1976} for elliptic and parabolic problems in 1970s, has recently shown great potential in solving the Helmholtz problem with high wave number \cite{wu2013,zhuwu2013,duwu2015,liwu2019}. The CIP-FEM uses the same approximation space as the FEM but modifies the sesquilinear form of the FEM by adding a least-squares term penalizing the jump of the normal derivative of the discrete solution at interior mesh interfaces. In this subsection, we introduce the CIP-FEM for the nonlinear PML problem \eqref{varnltpml}. Let $\EhI$ be the set of edges/faces of $\Th$ in $\Om$. For every $e=\partial K_1 \cap \partial K_2\in\EhI$, we define the jump $[v]$ of $v$ on $e$ as follows:
\begin{equation*}
[v]|_e:=v|_{K_1}-v|_{K_2}.
\end{equation*}
Now we introduce the energy space $V$ and the sesquilinear form $a_h (\cdot,\cdot)$ on $V\times V$ as
\begin{align*}
V &:= H_0^1(\D)\cap \prod_{K\in\Th} H^2(K), \\
a_h (u,v) &:= a(u,v)+J(u,v) \quad \forall u,v\in V, \\
J(u,v) &:= \sum\limits_{e\in\EhI} \gamma_e h_e\ine{[\na u\cdot n]}{[\na v\cdot n]}
\end{align*}
where the penalty parameters $\gamma_e$ for $e\in\EhI$ are numbers with nonpositive imaginary parts and $h_e:=\diam(e)$. It is clear that, if $\hu\in H^2(\OchO)$ is the solution to the nonlinear PML problem \eqref{eq:PML},
%\eqref{eq:PML:bound:all}, 
then $J(\hu,v) = 0$ for any $v\in V$. The CIP-FEM for \eqref{varnltpml} %\eqref{eq:PML:bound:all} 
reads as: find $u_h\in V_h$ such that
\begin{equation}\label{eq:cip-fem}
a_h(u_h,v_h)-k^2\vep\big(\abs{u_h+\uinc}^2(u_h+\uinc),v_h\big)_{\Omz} = (f,v_h)_{\Om}\quad \forall v_h\in V_h.
\end{equation}
%Similarly to \eqref{eq:dip}, we consider a Newton’s iterative procedure to approach this nonlinear CIP-FEM system: 
The Newton's iteration for solving the CIP-FEM \eqref{eq:cip-fem} reads: for a given $u_h^0\in V_h$, find ${u_h^{l+1}} \in V_h,~l=0,1,2,\cdots,$ such that
%\eqref{eq:cip-fem}: 

%For a given $u_h^0\in V_h$, find ${u_h^{l+1}} \in V_h,~l=0,1,2,\cdots,$ such that 
\begin{equation}\label{eq:dip0} 
\begin{aligned}
  a_h(u_h^{l+1},v_h)&- k^2\vep\big(2\,| u_h^l+\uinc |^2 u_h^{l+1} + (u_h^l+\uinc)^2 \overline{u_h^{l+1}},v_h\big)_{\Omz} \\
  = (f,v_h)_{\Om} &- k^2\vep\big(2\,|u_h^l+\uinc|^2 u_h^l - (u_h^l+\uinc)^2 \overline{\uinc},v_h\big)_{\Omz} %k^2\vep\big((u_h^{l-1}+\uinc)^2\overline{(u_h^{l-1}+\uinc)},v_h\big) \quad\forall v_h\in V_h. 
\end{aligned}
\end{equation}

\begin{remark} There are several important remarks about the CIP-FEM \eqref{eq:cip-fem}: 

(i) The CIP-FEM reduces to the standard FEM if we take $\gamma_e\equiv 0$.

(ii) The sesquilinear form $a(\cdot,\cdot)$ is coercive in the PML region $\hOm$ (cf. \cite[Lemma 3.6]{liwu2019}), and hence, the PML problem behaves more like an elliptic one. Based on this consideration, 
penalty terms in $J(\cdot,\cdot)$ are only added for those edges/faces in $\Om$ 
in order to reduce the pollution error.

(iii) By combining the analyses in subsection \ref{sec:pee} with the techniques in \cite{liwu2019}, 
we can also derive the error estimate \eqref{eq:err} for the CIP-FEM \eqref{eq:cip-fem} and the quadratic convergence \eqref{eq:dquad} 
for its Newton's iteration \eqref{eq:dip0}.

(iv) The modified Newton’s method for solving the CIP-FEM reads:
\begin{equation}\label{eq:dip2}
a_h(u_h^{l+1},v_h) - 2k^2\vep\big(| u_h^l+\uinc |^2 u_h^{l+1},v_h\big)_{\Omz}
= (f,v_h)_{\Om} - k^2\vep\big(| u_h^l+\uinc |^2 (u_h^l - \uinc),v_h\big)_{\Omz}\quad\forall v_h\in V_h,
\end{equation}
and the frozen-nonlinearity method for solving the CIP-FEM reads:
\begin{equation}\label{eq:dip1} 
    a_h({u_h^{l+1}},v_h)- k^2\vep\big( | {u_h^l}+\uinc | ^2 {u_h^{l+1}}, v_h\big)_{\Omz} = (f,v_h)_{\Om} + k^2\vep \big( | {u_h^l}+\uinc | ^2 \uinc, v_h \big)_{\Omz} \quad\forall v_h\in V_h. 
  \end{equation}
\end{remark}

\section{Numerical results} \label{s:5}
In this section, we simulate the NLH \eqref{eq:Helm}--\eqref{eq:Somm} with $\Om=\B_1$ and $\Omz=\B_{1/2}$. The problem is first truncated by the PML technique and then discretized by the linear (CIP-)FEM. We take the following penalty parameter 
\begin{equation}\label{eq:pp}
  \gamma_e= -\frac{\sqrt{3}}{24}-\frac{\sqrt{3}}{1728}(k h_e)^2
\end{equation}
for CIP-FEM, which is obtained by a dispersion analysis for 2D problems on equilateral triangulations \cite{hw12}. %Moreover, we set the PML parameter $\si_0=10$ and PML thickness $L=1/4$ which satisfy the condition \eqref{eq:assumption}.
The stop criterion used in the iterations is 
\begin{equation}\label{eq:tol}
\frac{\he{u_h^l - u_h^{l-1}}}{\he{u_h^l}} <  \mbox{tol} =10^{-6}, \quad\text{for some } l\geq 1. 
\end{equation}

\subsection{Accuracy and pollution effect}
We choose the exact solution (scattered field) $u$ (cf. \cite{liwu2019}) and incident wave $\uinc$ to be
\[u = \begin{cases}
 \frac{\i\pi}{2k}H_1^{(1)}(k)J_0(kr)-\frac{1}{k^2},&\text{in }\Om,\\
 \frac{\i\pi }{2k} J_1(k) H_0^{(1)}(kr),&\text{otherwise},
 \end{cases}\qaq
 \uinc= \frac{J_0(kr)}{k^{1.5}},
\]
respectively. The Kerr constant is chosen as $\vep=k^{-2}$ and $f$ satisfies the equation \eqref{eq:Helm}. 
The PML parameter and PML thickness are set by $\si_0=4$ and $L=1$ which satisfy the condition \eqref{eq:assumption}.
The left graph of Figure \ref{fig:nlex1} plots the relative $H^1$-errors of the FEM, CIP-FEM and FE interpolations for $k=10, 50$, and $100$, respectively. As is shown, for $k=10$, both the error curves of FE and CIP-FE solutions fit that of the FE interpolation very well, which indicate that the pollution effects do not work for small wave number. While for large $k$, e.g., $k=50$, the errors of FEM oscillate around $100\%$ before decaying in a range of mesh sizes far from the decaying point of the corresponding FE interpolations, and even farther for  $k=100$. The errors of CIP-FE solutions behave similarly, but begin to decay much earlier than FEM, which implies that the CIP-FEM reduces the pollution effect greatly.
Next, we fix $kh=\pi/5$, which implies that about $10$ degrees of freedom are set per wave-length, and then plot the relative $H^1$-errors of the FE solutions, the CIP-FE solutions, and the FE interpolations for increasing wave numbers $k$ in one figure (see the right graph of Figure \ref{fig:nlex1}). 
It is obvious that the pollution effect of FEM appears when $k$ becomes greater than some value less than 50, while the CIP-FEM is almost pollution-free for $k$ up to 200. Compared with FEM, CIP-FEM does effectively reduce the pollution error. 
%In addition, the errors of the FE interpolation and CIP-FE solution for the case of $kh = 1$ are almost twice as big as those for the case of $kh = 0.5$, which indicates that the relative $H^1$-error of the CIP-FEM behaves like the interpolation error $O(kh)$ when $k\leq 200$.
\begin{figure}[tbp]
\begin{minipage}[c]{0.5\textwidth}
\centering
\includegraphics[height=5.5cm]{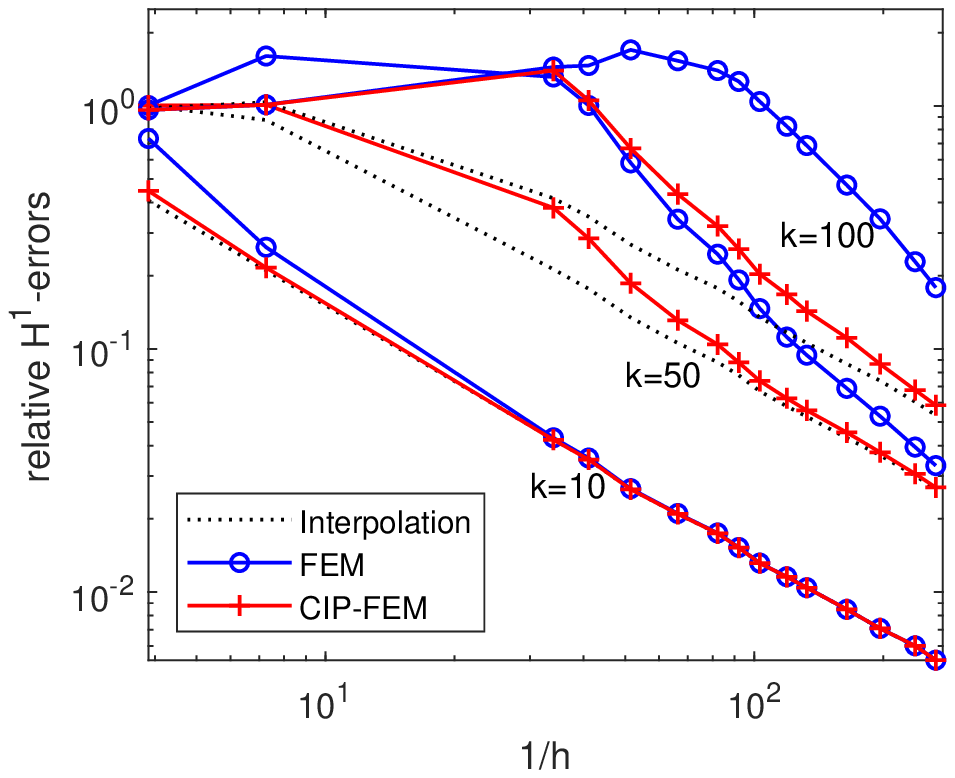}
\end{minipage}%
\begin{minipage}[c]{0.5\textwidth}
\centering
\includegraphics[height=5.5cm]{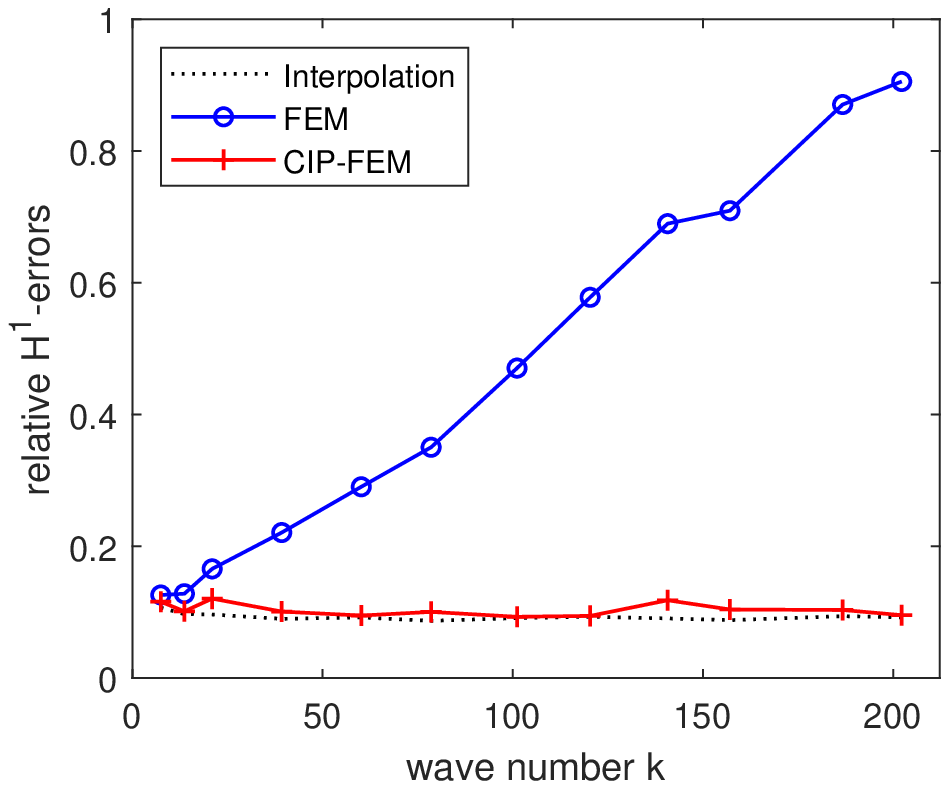}
\end{minipage}
\caption{\itshape{The relative $H^1$-errors of the FE interpolation, the FE solution, and the CIP-FE solution.}}
\label{fig:nlex1}
\end{figure}

\subsection{Optical bistability}
Optical bistability (see e.g. \cite{boyd2008}) refers to the situation in which two different output intensities are possible for a given input intensity. It can be used as a switch in optical communication and in optical computing. We consider the NLH with $k=k_0=5.4$ in $\Omz^c$ and $k=k_1=3.5k_0$ in $\Omz$ (cf. \cite{wuzou2018,yuan2017}). Set $h=10^{-2}$ and Kerr constant $\vep=10^{-12}$, the incident wave $\uinc=Ie^{\i k_0x}$ and source term $f_0=0$, that is 
\[f=
 \begin{cases}
 (k_1^2-k_0^2)\uinc, & \text{in } \Omz,\\
 0, & \text{otherwise}.
 \end{cases}
\]
In this example, we set $\si_0=10$ and $L=1/4$ to reduce the computational area. 
Figure \ref{fig:nlex3_1} plots the energy norm of the scattered field computed by the Newton's method \eqref{eq:dip0} versus that of the incident wave $\uinc$. A reference incident wave $\uinc^0=I_0 e^{\i k_0x}$ with $I_0=10^5$ is introduced for enhancing the  nonlinear effect. Obviously, the larger the amplitude $I$, the stronger the intensity of the incident wave. 
The energy of the scattered field jumps to the upper branch from the lower branch as the intensity of the incident wave increases to $I\approx 264651$, and falls down from the upper branch to the lower branch as the intensity decreases to $I \approx 241294$.
%\cb{The lower branch (the solid line) is computed by the iterative method \eqref{eq:dip1}, which becomes invalid when $A$ is larger than some constant. However, the modified Newton's method \eqref{eq:dip2} behaves more robust for large range of $A$, but it jumps to the upper branch as the intensity of the incident wave increases to $A\approx 320380$, and fall down to the lower branch as the intensity decreases to $A\approx 269620$, as the dot line shown in the figure. The middle branch (the dot-dash line) is computed by the standard Newton's method \eqref{eq:dip0}.} 
As shown in the figure, for $241294< I < 264651$, the NLH has three different solutions, where the two solutions in the upper and lower branches are presumably stable, and the solution in the middle branch is unstable. This phenomenon corresponds to the optical bistability. For $I=255000$, the electric field patterns of these three solutions corresponding to the three circled points in Figure~\ref{fig:nlex3_1} are shown in Figure \ref{fig:nlex3_2}. As expected, the nonlinear phenomenon of optical bistability has been successfully simulated. 

Finally, we compare  the convergence rates of all the three methods \eqref{eq:dip0}--\eqref{eq:dip1} by solving the solution in the lower branch at $I=263000$ with the same initial values of zero. We use the solution computed by the Newton's method \eqref{eq:dip0} with a small tolerance $\mbox{tol}=10^{-13}$ (see  \eqref{eq:tol}) as the ``exact" CIP-FE solution $u_h$. The relative error and convergence order are defined by 
\[e_h^l:= \frac{\He{u_h^{l}-u_h}}{\He{u_h}} \qaq \mbox{order} := \frac{\log e_h^{l+1} - \log e_h^l}{\log e_h^l - \log e_h^{l-1}}, \quad l\geq 1.
\]
The numerical results are listed in Table~\ref{tab:steps}, which shows, as expected, that the Newton’s method converges quadratically while the other two methods converge linearly. There is no doubt, the Newton’s method converges much faster.

\begin{figure}[tbp]
\centering
\includegraphics[width=6cm]{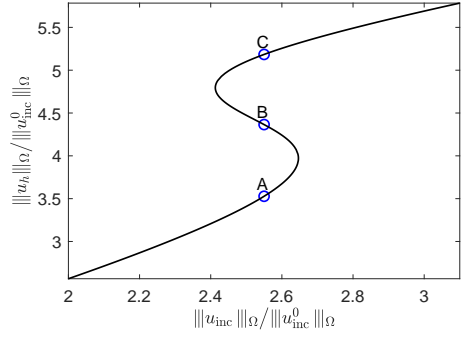}
\caption{\itshape{X-axis: incident wave; Y-axis: scattered field. $I$ is $200000:310000$.}}
\label{fig:nlex3_1}
\end{figure}
\begin{figure}[t]
\centering
\includegraphics[width=13cm]{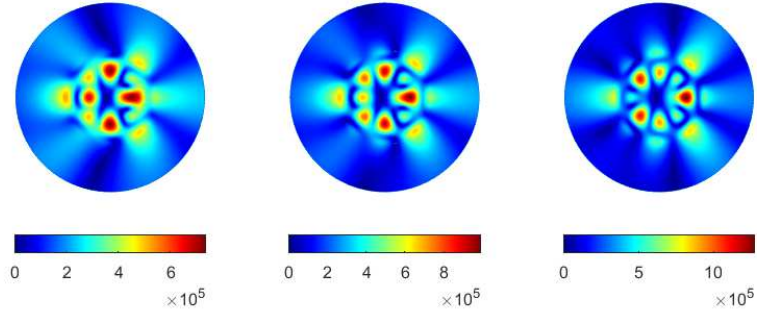}
\caption{\itshape{Scattered field patterns of the three solutions marked as the points {\rm A}, {\rm B}, and {\rm C} in Figure \ref{fig:nlex3_1}.}}
\label{fig:nlex3_2}
\end{figure}

%\begin{table}[t]
%  \centering
%  \cb{
%  \begin{tabular}{|c|c|c|c|c|c|c|c|c|c|c|}
%  \hline 
%  \multicolumn{2}{|c|}{step} & 2 & 3 & 4 & 5 & 6 & 10 & 18 & 24 \\
%  \hline
%  \multirow{2}*{Newton} & error & 2.40e-2 & 1.40e-3 & 5.90e-6 & 3.37e-9 & 2.08e-12 &\myslbox &\myslbox  &\myslbox   \\
%  \cline{2-10}
%  & order & -  & 1.83 & 1.95 & 1.36 & 0.99  &\myslbox  &\myslbox &\myslbox \\
%  \hline
%  \multirow{2}*{modified Newton} & error  & 4.78e-2 & 2.20e-2 & 1.05e-2 & 5.10e-3 & 2.50e-3  & 1.40e-4 & 4.62e-7 & \myslbox \\
%\cline{2-10}
%& order  & -  & 0.91 & 0.96 & 0.98 & 0.99 & 1.00 & 1.00 & \myslbox \\
%\hline
%\multirow{2}*{frozen-nonlinearity} & error & 6.33e-2 & 3.64e-2 & 2.15e-2 & 1.29e-2 & 7.80e-3 & 1.10e-3  & 2.20e-5 & 1.18e-6 \\
%\cline{2-10}
%& order  & - & 0.96 & 0.95 & 0.97 & 0.98 & 1.00 & 1.00 & 1.00 \\
%\hline
%  \end{tabular}
%  \caption{\it Relative errors and convergence orders of the iterative methods. %$/$ represents iteration termination. $0$ and $-1$ represent iteration success and iteration failure, respectively.
%  }
%  \label{tab:steps}}
%\end{table}

\begin{table}[H]
  \centering
    \scalebox{0.98}{\begin{tabular}{|c|c|c|c|c|c|c|c|c|c|c|}
  \hline 
  \multicolumn{2}{|c|}{step} & 1 & 2 & 3 & 4 & 5 & 6 & 82 & 120 \\
  \hline
  \multirow{2}*{Newton} & error & 7.48e-2 & 2.85e-3 & 4.65e-3 & 1.40e-4 & 1.27e-7 & 9.36e-14  &\myslbox  &\myslbox   \\
  \cline{2-10}
  & order & -  & 0.37 & 1.88 & 1.93 & 2.00  & 2.02  &\myslbox  &\myslbox \\
  \hline
  \multirow{2}*{modified Newton} & error  & 4.11e-2 & 2.91e-2 & 2.05e-2 & 1.53e-2 & 1.12e-2  & 8.45e-3 & 9.10e-12  & \myslbox \\
\cline{2-10}
& order  & -  & 0.11 & 1.01 & 0.84 & 1.07 & 0.91 & 1.00  & \myslbox \\
\hline
\multirow{2}*{frozen-nonlinearity} & error & 1.72e-1 & 1.19e-1 & 8.53e-2 & 6.34e-2 & 4.84e-2 & 3.76e-2  & 1.35e-8 & 9.22e-12 \\
\cline{2-10}
& order  & - & 0.21 & 0.90 & 0.89 & 0.91 & 0.93 & 1.00 & 1.00 \\
\hline
  \end{tabular}}
  \caption{\it Relative errors and convergence orders of the iterative methods. %$/$ represents iteration termination. $0$ and $-1$ represent iteration success and iteration failure, respectively.
  }
  \label{tab:steps}
\end{table}

\appendix
%\section{A discrete stability estimate} 
\renewcommand{\theequation}{A.\arabic{equation}}
\renewcommand{\thetheorem}{A.\arabic{theorem}}
\renewcommand{\thefigure}{A.\arabic{figure}}
\setcounter{equation}{0}
\setcounter{theorem}{0}
\setcounter{figure}{0}
\section*{Appendix: A discrete stability estimate} 
\begin{lemma}\label{app:dstab} Given complex-valued functions $p,\,q\in L^{\infty}(\D)$, and $\tilde g \in L^2(\D)$,  
consider $\tilde w\in H_0^1(\D)$ solving 
\eq{\label{app:tw} -\nabla\cdot (A\nabla \tilde w) + p\tilde w + \overline{q \tilde w} = \tilde g}
in the weak sense. Suppose there exists a positive constant $\tilde c_0$ such that the following stability estimate holds:
\eq{\label{app:stab:tw} \norm{\tilde w}_0  \leq \tilde c_0 \norm{\tilde g}_0.}
Let $w_h\in V_h$ be the finite element solution to
\eq{\label{app:wh} (A\nabla w_h, \nabla v_h) + (p w_h, v_h) + (q \overline{w_h}, v_h) = (g, v_h) \quad \forall\, v_h \in V_h.}
Then the FE solution satisfies the discrete stability estimates
\eq{\label{app:stab:wh} \norm{w_h}_0\ls \tilde c_0\norm{g}_0\quad\text{and}\quad  \norm{\na w_h}_0 \ls\big(\norm{p}_{L^\infty(\D)} + \norm{q}_{L^\infty(\D)} + \tilde c_0^{-1} \big)^{\frac12} \tilde c_0\norm{g}_0,}
under the condition that the following quantity is sufficiently small:
\eq{\label{app:h0} \big(1+ (\norm{p}_{L^\infty(\D)} + \norm{q}_{L^\infty(\D)} ) \tilde c_0\big) \max\big\{ \norm{p}_{L^\infty(\D)} + \norm{q}_{L^\infty(\D)}, \tilde c_0^{-1} \big\}h^2\,.}
 %\text{ is sufficiently small,}\notag}

\end{lemma}
\begin{proof} First, applying  \cite[Theorem 4.5]{huangzou2007} and using \eqref{app:stab:tw}, we obtain the following regularity estimate:
\eqn{
  \norm{\tilde w}_{2,\OchO} \ls \norm{\tilde g-p\tilde w -\overline{q \tilde w}}_0\ls\big (1+\big(\norm{p}_{L^\infty(\D)} + \norm{q}_{L^\infty(\D)} \big) \tilde c_0\big)\norm{\tilde g},
}
that is, there exists a positive constant $\tilde c_2\eqsim 1+\big(\norm{p}_{L^\infty(\D)} + \norm{q}_{L^\infty(\D)} \big) \tilde c_0$ such that
\eq{\label{app:stab:tw2}
  \norm{\tilde w}_{2,\OchO}\le \tilde c_2\norm{\tilde g}.}
We use the convention $v = v_r +\i v_i$ for all complex-valued functions, where $v_r$ and $v_i$ are both real-valued. 
Choosing $v_h :\D \to \R$ be any real function in \eqref{app:wh} and taking the real and imaginary parts, we get
\begin{align}
(A_r \na w_{hr} - A_i \na w_{hi}, \na v_h) + (p_r w_{hr} - p_i w_{hi} + q_r w_{hr} + q_i w_{hi}, v_h) &= (g_r, v_h), \label{app:wh1} \\
(A_i \na w_{hr} + A_r \na w_{hi}, \na v_h) + (p_i w_{hr} + p_r w_{hi} + q_i w_{hr} - q_r w_{hi}, v_h) &= (g_i, v_h). \label{app:wh2}
\end{align}
Let $z\in H_0^1(\D)$ solve the dual problem $-\na \cdot (\overline{A} \na z) + \overline{p}z + q\overline{z} = w_h$ in $\D$, which can be rewritten as
\begin{align}
-\na \cdot (A_r \na z_r) - \na \cdot (A_i \na z_i) + (p_r z_r + p_i z_i) + (q_r z_r + q_i z_i) &= w_{hr}, \label{app:z1}\\
\na \cdot (A_i \na z_r) - \na \cdot (A_r \na z_i) + (p_r z_i - p_i z_r) + (q_i z_r - q_r z_i) &= w_{hi}.\label{app:z2}
\end{align}
Testing \eqref{app:z1} and \eqref{app:z2} by $w_{hr}$ and $w_{hi}$, respectively, and using \eqref{app:wh1}--\eqref{app:wh2} and the fact that $A$ is symmetric, 
\eqn{\norm{w_h}_0^2 =&\; (A_r \na z_r + A_i \na z_i, \na w_{hr}) + (p_r z_r + p_i z_i + q_r z_r + q_i z_i, w_{hr}) \\
 &\; - (A_i \na z_r - A_r \na z_i, \na w_{hi}) + (p_r z_i - p_i z_r + q_i z_r - q_r z_i, w_{hi}) \\
=&\; (A_r \na w_{hr} - A_i \na w_{hi}, \na z_r) +  (p_r w_{hr} - p_i w_{hi} + q_r w_{hr} + q_i w_{hi}, z_r) \\
 &\; + (A_i \na w_{hr} + A_r \na w_{hi}, \na z_i) + (p_i w_{hr} + p_r w_{hi} + q_i w_{hr} - q_r w_{hi}, z_i) \\
=&\; (A_r \na w_{hr} - A_i \na w_{hi}, \na (z_r - (P_h z)_r)) +  (p_r w_{hr} - p_i w_{hi} + q_r w_{hr} + q_i w_{hi}, z_r - (P_h z)_r) \\
 &\; + (A_i \na w_{hr} + A_r \na w_{hi}, \na (z_i - (P_h z)_i)) + (p_i w_{hr} + p_r w_{hi} + q_i w_{hr} - q_r w_{hi}, z_i - (P_h z)_i) \\
 &\; + (g_r, (P_h z)_r - z_r) + (g_r, z_r) + (g_i, (P_h z)_i - z_i) + (g_i, z_i) \\
=&\; (p_r w_{hr} - p_i w_{hi} + q_r w_{hr} + q_i w_{hi}, z_r - (P_h z)_r) + (p_i w_{hr} + p_r w_{hi} + q_i w_{hr} - q_r w_{hi}, z_i - (P_h z)_i) \\
 &\; + (g_r, (P_h z)_r - z_r) + (g_r, z_r) + (g_i, (P_h z)_i - z_i) + (g_i, z_i),
}
where $P_h$ is the elliptic projection defined by \eqref{eq:Ph} and we used $(A\na w_h, \na (z-P_h z)) = 0$ to derive the last equality. Noting that $\overline{z} \in H_0^1(\D)$ is the solution to \eqref{app:tw} with $\tilde g = \overline{w_h}$, %$-\na \cdot (A \na \overline{z}) + p\overline z + \overline{q} z = \overline{w_h}$, 
we deduce from \eqref{app:stab:tw} and \eqref{app:stab:tw2} that 
\eqn{\norm{z}_0\le \tilde c_0 \norm{w_h}_0\quad\text{and}\quad \norm{z}_{2,\OchO} \leq \tilde c_2\norm{w_h}_0.}
Therefore, the error estimate \eqref{eq:errPh} yields
\eqn{\norm{w_h}_0^2 \leq&\; C_{\mathcal E} \tilde c_2 h^2 \big( (\norm{p}_{L^\infty(\D)} + \norm{q}_{L^\infty(\D)}) \norm{w_h}_0^2  +   \norm{g}_0 \norm{w_h}_0 \big) +  \tilde c_0 \norm{g}_0 \norm{w_h}_0, 
}
where $C_{\mathcal E}$ is the invisible constant in \eqref{eq:errPh}. Letting  
\eqn{h_0^2 = \min \big\{ \big( 2 C_{\mathcal E}\tilde c_2  (\norm{p}_{L^\infty(\D)} + \norm{q}_{L^\infty(\D)}) \big)^{-1}, (2 C_{\mathcal E}\tilde c_2 )^{-1}\tilde c_0 \big\}, }
we readily get the first estimate in \eqref{app:stab:wh} for $0<h\leq h_0$.
%\eq{\label{app:l2} \norm{w_h}_0 \leq 3 \tilde c_0 \norm{g}_0.}

On the other hand, letting $v_h=w_h$ in \eqref{app:wh} and using the derivations of \eqref{h12d}--\eqref{h13d} give
\eqn{\norm{\na w_h}_0^2 &\ls \Re (A\na w_h,\na w_h)
\ls \big(\norm{p}_{L^\infty(\D)} + \norm{q}_{L^\infty(\D)} \big) \norm{w_h}^2_0+\norm{g}_0\norm{w_h}_0\\
&\ls \big(\norm{p}_{L^\infty(\D)} + \norm{q}_{L^\infty(\D)} +  \tilde c_0^{-1}\big) \tilde c_0^2 \norm{g}_0^2.
}
%which together with \eqref{app:l2} yields 
%\eqn{\norm{w_h}_1\cb{\ls \tilde c_0 } \Big(1 + \big(\norm{p}_{L^\infty(\D)} + \norm{q}_{L^\infty(\D)} + \cb{ \tilde c_0^{-1}}\big)^{\frac12} \Big)\norm{g}_0.}
Hence, \eqref{app:stab:wh} follows and the proof of this lemma is completed.
\end{proof}

%%%%%%%%%%%%%%%%%%%%%%% Documents End %%%%%%%%%%%%%%%%%%%%%%%%%%

\bibliographystyle{plain}
\bibliography{NLHPML}

\end{document}

%% file: NLHPML_def.tex
\newcommand{\abs}[1]{\left\vert #1 \right\vert}
\newcommand{\norm}[1]{\left\Vert #1\right\Vert}
\newcommand{\He}[1]{\left\Vert{\hskip -2.7pt}\left\vert #1 \right\vert{\hskip -2.7pt}\right\Vert}
\newcommand{\he}[1]{\big\vert\kern-0.25ex\big\vert\kern-0.25ex\big\vert #1 \big\vert\kern-0.25ex\big\vert\kern-0.25ex\big\vert}
\newcommand{\Lt}[2]{\left\Vert #1\right\Vert_{0,#2}}
\newcommand{\LtD}[1]{\left\Vert #1\right\Vert_{0}}
\newcommand{\Ho}[2]{\left\Vert #1\right\Vert_{1,#2}}

\newcommand{\Ht}[2]{\left\Vert #1\right\Vert_{2,#2}}
\newcommand{\sHt}[2]{\left\vert #1\right\vert_{2,#2}}
\newcommand{\ine}[2]{\langle #1,#2 \rangle_e}
\newcommand{\inOm}[2]{(#1,#2)_{\Omega}}
\newcommand{\Linf}[1]{\norm{#1}_{L^{\infty}(\Omz)}}
\newcommand{\B}{\mathcal{B}}
\newcommand{\D}{\mathcal{D}}
\newcommand{\R}{\mathbb{R}}
\newcommand{\Z}{\mathbb{Z}}
\newcommand{\N}{\mathbb{N}}
\newcommand{\Cm}{\mathbb{C}}
\newcommand{\EhI}{\mathcal{E}_h^I}
\newcommand{\Th}{\mathcal{T}_h}

\newcommand{\Jn}{J_{n}}

\newcommand{\Hn}{H_n^{(1)}}

\newcommand{\hu}{\hat{u}}
\newcommand{\hR}{\hat{R}}
\newcommand{\hC}{\hat{C}}

\newcommand{\tr}{\tilde{r}}
\newcommand{\tu}{\tilde{u}}

\newcommand{\ttn}{\tilde{t}}
\newcommand{\thR}{\tilde{\hat{R}}}
\newcommand{\Om}{\Omega}
\newcommand{\hOm}{\hat{\Omega}}
\newcommand{\hGamma}{\hat{\Gamma}}
\newcommand{\cni}{C_{\mathrm{Nir}}}

\newcommand{\intRhR}{\int_R^{\hR}}
\newcommand{\intrR}{\int_r^R}

\newcommand{\intr}{\int_0^{r}}

\newcommand{\inthR}{\int_0^{\hR}}

\newcommand{\al}{\alpha}
\newcommand{\be}{\beta}
\newcommand{\si}{\sigma}
\newcommand{\siz}{\sigma_0}
\newcommand{\de}{\delta}

\newcommand{\na}{\nabla}
\newcommand{\Ga}{\Gamma}
\newcommand{\pa}{\partial}
\newcommand{\ta}{\theta}
\newcommand{\vp}{\varphi}
\newcommand{\supp}{\mathrm{supp}\, }
\renewcommand{\i}{{\rm\mathbf i}}

\newcommand{\OchO}{\Omega\cup\hat{\Omega}}

\newcommand{\cb}[1]{{\color{blue}#1}}

\newcommand{\diam}{\mathrm{diam}\,}
\newcommand{\ls}{\lesssim}
\newcommand{\gs}{\gtrsim}

\newcommand{\mL}{\mathcal{L}}

\newcommand{\uinc}{u_{\rm inc}}

\newcommand{\oneo}{{\bf 1}_{\Omz}}
\newcommand{\Mf}{M(f)}
\newcommand{\vep}{\varepsilon}
\newcommand{\Omz}{\Omega_{0}}
\newcommand{\dist}{\mathrm{dist}\,}
\newcommand{\EPML}{\mathcal{E}^{\rm PML}}

\newcommand{\anl}{a^{\rm NL}}
\newcommand{\cu}{\check{u}}

\newcommand{\eq}[1]{\begin{align}#1\end{align}}
\newcommand{\eqn}[1]{\begin{align*}#1\end{align*}}

\newcommand{\Ap}[1]{\left\{ #1 \right\}} % { * }

\newcommand{\BgA}[1]{\Big\{ #1 \Big\}}

\newcommand{\Sp}[1]{\left( #1 \right)} % ( * )
\newcommand{\bgS}[1]{\big( #1 \big)}
\newcommand{\BgS}[1]{\Big( #1 \Big)}

\newcommand{\qaq}{\quad\mbox{and}\quad}